\documentclass[11pt,a4paper]{article}

\usepackage[T1]{fontenc}
\usepackage[utf8]{inputenc}
\usepackage{amsmath,amssymb,amsthm,mathtools,mathrsfs}
\usepackage{libertinus}
\usepackage{microtype}
\usepackage{booktabs,tabularx,array}
\usepackage[
  a4paper,
  textwidth=150mm,
  textheight=235mm,
  heightrounded,
  centering
]{geometry}
\usepackage[dvipsnames]{xcolor}
\usepackage{titlesec}
\usepackage{etoolbox}
\usepackage{cite}
\usepackage[
  colorlinks=true,
  linkcolor=black,
  citecolor=black,
  urlcolor=black,
  pdftitle={Subexponential Bohnenblust--Hille Inequalities on Finite Cyclic Groups},
  pdfauthor={Daniel Marinho Pellegrino and Anselmo Raposo Junior},
  pdfkeywords={Bohnenblust--Hille inequality, finite cyclic groups, interaction order, subexponential constants, Potts hypercontractivity, Blei mixed norms, polarization, Bohr radius}
]{hyperref}

\allowdisplaybreaks
\numberwithin{equation}{section}
\titleformat{\section}[block]
  {\centering\normalfont\large\scshape}
  {\thesection.}{0.65em}{}
\titlespacing*{\section}{0pt}{3.2ex plus .8ex minus .2ex}{1.8ex plus .2ex}
\titleformat{\subsection}[block]
  {\normalfont\bfseries}
  {\thesubsection.}{0.55em}{}
\titlespacing*{\subsection}{0pt}{2.7ex plus .6ex minus .2ex}{1.1ex plus .2ex}
\titleformat{\subsubsection}[runin]
  {\normalfont\itshape}
  {\thesubsubsection.}{0.5em}{}[.]
\titlespacing*{\subsubsection}{0pt}{2.0ex plus .4ex minus .2ex}{0.7em}

\newtheoremstyle{cleanplain}
  {7pt}{7pt}{\itshape}{}{\bfseries}{.}{0.5em}{}
\newtheoremstyle{cleandefinition}
  {7pt}{7pt}{\normalfont}{}{\bfseries}{.}{0.5em}{}
\newtheoremstyle{cleanremark}
  {7pt}{7pt}{\normalfont}{}{\itshape}{.}{0.5em}{}

\usepackage{aliascnt}
\theoremstyle{cleanplain}
\newtheorem{maintheorem}{Theorem}

\newtheorem{theorem}{Theorem}[section]
\newaliascnt{proposition}{theorem}
\newtheorem{proposition}[proposition]{Proposition}
\aliascntresetthe{proposition}
\newaliascnt{lemma}{theorem}
\newtheorem{lemma}[lemma]{Lemma}
\aliascntresetthe{lemma}
\newaliascnt{corollary}{theorem}
\newtheorem{corollary}[corollary]{Corollary}
\aliascntresetthe{corollary}
\theoremstyle{cleandefinition}
\newaliascnt{definition}{theorem}

\aliascntresetthe{definition}
\theoremstyle{cleanremark}
\newaliascnt{remark}{theorem}
\newtheorem{remark}[remark]{Remark}
\aliascntresetthe{remark}

\newcommand{\T}{\mathbb T}
\newcommand{\C}{\mathbb C}
\newcommand{\N}{\mathbb N}
\newcommand{\E}{\mathbb E}
\newcommand{\supp}{\operatorname{supp}}
\newcommand{\Ind}{\textcolor{black}{\operatorname{Ind}}}
\newcommand{\Indup}{\textcolor{black}{\operatorname{Ind}_{\uparrow}}}
\newcommand{\BHdeg}[2]{\mathrm{BH}^{\mathrm{deg}}_{#1,#2}}
\newcommand{\BHint}[2]{\mathrm{BH}^{\mathrm{int}}_{#1,#2}}

\makeatletter
\newcommand{\@address}{}
\newcommand{\@email}{}
\newcommand{\@subjclass}{}
\newcommand{\@keywords}{}
\newcommand{\@status}{}
\newcommand{\address}[1]{\gdef\@address{#1}}
\newcommand{\email}[1]{\gdef\@email{#1}}
\newcommand{\subjclass}[2][]{\gdef\@subjclass{#2}}
\newcommand{\keywords}[1]{\gdef\@keywords{#1}}
\newcommand{\status}[1]{\gdef\@status{#1}}
\renewcommand{\maketitle}{%
  \begin{center}
    \vspace*{0.4em}
    {\LARGE\scshape \@title\par}
    \vspace{1.25em}
    {\normalsize\scshape \@author\par}
    \vspace{0.45em}
    {\small \@address\par}
    \ifdefempty{\@email}{}{\vspace{0.2em}{\small\texttt{\@email}\par}}
    \ifdefempty{\@status}{}{\vspace{0.55em}{\small\itshape\color{black}\@status\par}}
  \end{center}
  \vspace{1.0em}
}
\renewenvironment{abstract}
  {\begin{center}\begin{minipage}{0.91\textwidth}\small
   \noindent\textsc{Abstract.}\enspace}
  {\end{minipage}\end{center}\vspace{0.65em}}
\newcommand{\makefrontmatterdetails}{%
  \begin{center}
  \begin{minipage}{0.91\textwidth}
  \small\noindent\textsc{Keywords.}\enspace\@keywords\par
  \vspace{0.25em}
  \noindent\textsc{2020 Mathematics Subject Classification.}\enspace\@subjclass
  \end{minipage}
  \end{center}
  \vspace{0.75em}
}
\makeatother

\title{Subexponential Bohnenblust--Hille Inequalities\\on Finite Cyclic Groups}
\author{Daniel Marinho Pellegrino$^{1,*}$ and Anselmo Raposo Junior$^{2}$}
\address{$^{1}$Department of Mathematics, Universidade Federal da Para\'iba, Campus I, Cidade Universit\'aria, Jo\~ao Pessoa, PB 58051--900, Brazil\\
$^{2}$Department of Mathematics, Universidade Federal do Maranh\~ao, Cidade Universit\'aria Dom Delgado, Avenida dos Portugueses 1966, S\~ao Lu\'is, MA 65080--805, Brazil}
\email{$^{*}$Corresponding author: \href{mailto:dmpellegrino@gmail.com}{dmpellegrino@gmail.com}}
\status{}

\begin{document}
\maketitle

\begin{abstract}
Let $C_q$ denote the group of the $q$th roots of unity.  Becker, Klein,
Slote, Volberg and Zhang asked whether the dimension-free
Bohnenblust--Hille constants for functions on $C_q^N$ grow
subexponentially with the degree.  We answer this question affirmatively.
In fact, we prove a stronger estimate for functions whose Fourier characters
involve at most $d$ coordinates.  If $\BHint{d}{q}$ is the optimal constant
for this larger class, then, for every fixed $q\geq2$,
\[
 \BHint{d}{q}\leq
 \exp\left(
 c_q\sqrt{d\log d}
 +O_q\left(\sqrt{\frac d{\log d}}\log\log d\right)
 \right),
\]
where $c_2=2$ and
$c_q=\sqrt{2q\log(q-1)/(q-2)}$ for $q\geq3$.
As an application, we obtain two-sided estimates for the Bohr radius of the
Fourier layer formed by characters involving exactly $d$ coordinates, and
we determine its asymptotic behaviour in natural joint regimes of $d$ and
$N$.
\end{abstract}
\makefrontmatterdetails
\clearpage

\begingroup
\setcounter{tocdepth}{2}
\small
\tableofcontents
\endgroup
\clearpage

\section{Introduction}\label{sec:introduction}

{\color{black}
The Bohnenblust--Hille inequality is one of the basic dimension-free
coefficient estimates in analysis.  In its classical polynomial form, the
number of variables is allowed to grow while the exponent of the coefficient
norm depends only on the degree.  Over the last two decades, the quantitative
question of how the best constant depends on that degree has become part of
the problem itself.  For complex polynomials on the polytorus, and later for
functions on the Boolean cube, the decisive phenomenon is that the optimal
constants can be taken to grow subexponentially; see, for instance,
\cite{BPSS,DMP}.  This improvement is not merely cosmetic: the size of the
Bohnenblust--Hille constant enters directly in applications where one must
keep estimates uniform in the ambient dimension.

Products of finite cyclic groups form a natural intermediate setting, but one
in which the usual arguments no longer pass unchanged.  Slote, Volberg and
Zhang established the dimension-free Bohnenblust--Hille inequality on these
groups \cite{SVZ}.  Becker, Klein, Slote, Volberg and Zhang subsequently gave
an explicit exponential bound by a dimension-free discretization argument
and asked whether the dependence on the degree could in fact be made
subexponential \cite{BKSVZ}.  The purpose of the present paper is to answer
that question, and to do so for the larger class in which complexity is
measured by the number of active coordinates rather than by total degree.
The definitions below are arranged to make this distinction precise before
we state the main result.
}

We begin with the classical setting and fix the notation that will be used
throughout the paper.  The unit circle is
\begin{equation}\label{eq:torus-intro}
 \T:=\{z\in\C:|z|=1\},
\end{equation}
and $\T^N$ is the $N$-dimensional torus, or polytorus.  If
$\alpha=(\alpha_1,\ldots,\alpha_N)$ is a multi-index with nonnegative
integer entries, we write
\[
 |\alpha|:=\alpha_1+\cdots+\alpha_N,
 \qquad
 z^\alpha:=z_1^{\alpha_1}\cdots z_N^{\alpha_N}.
\]
Thus a complex polynomial of degree at most $d$ in $N$ variables is a map
\[
 P:\C^N\longrightarrow\C
\]
of the form
\[
 P(z)=\sum_{|\alpha|\leq d}c_\alpha z^\alpha,
 \qquad z\in\C^N.
\]
We shall use its supremum norm on the torus,
\begin{equation}\label{eq:torus-supnorm-intro}
 \|P\|_{L_\infty(\T^N)}
 :=\sup_{z\in\T^N}|P(z)|.
\end{equation}

For $d\geq1$, put
\begin{equation}\label{eq:pd-intro}
 p_d:=\frac{2d}{d+1}.
\end{equation}
The polynomial Bohnenblust--Hille inequality \cite{BH} asserts that there is a
constant $B_d$, depending on $d$ but not on $N$, such that
\begin{equation}\label{eq:classical-BH-intro}
 \left(\sum_{|\alpha|\leq d}|c_\alpha|^{p_d}\right)^{1/p_d}
 \leq B_d\,\|P\|_{L_\infty(\T^N)}.
\end{equation}
We now pass to finite cyclic groups.  Fix an integer $q\geq2$, let
\begin{equation}\label{eq:Cq-intro}
 C_q:=\{1,\omega,\ldots,\omega^{q-1}\}\subset\T,
 \qquad \omega:=e^{2\pi i/q},
\end{equation}
and give $C_q$ the group law inherited from multiplication in $\C$.
The dimension-free Bohnenblust--Hille problem on the products $C_q^N$ was
studied by Slote, Volberg and Zhang in \cite{SVZ}.  {\color{black}Note that} $C_2=\{-1,1\}$, and $C_2^N$ is the Boolean cube.  Let
$\mu_q$ denote normalized counting measure on $C_q$; its product measure on
$C_q^N$ is $\mu_q^N$.  For a function $f:C_q^N\to\C$, we write
\begin{equation}\label{eq:cyclic-supnorm-intro}
 \|f\|_\infty:=\max_{x\in C_q^N}|f(x)|.
\end{equation}

Every $f:C_q^N\to\C$ has a unique Fourier expansion.  Write
$\mathbb Z_q:=\mathbb Z/q\mathbb Z$ and identify its elements with the
representatives $\{0,1,\ldots,q-1\}$.  We use these
representatives to index the characters.  Thus
\begin{equation}\label{eq:Fourier-intro}
 f(x)=\sum_{\alpha\in\{0,\ldots,q-1\}^N}
 \widehat f(\alpha)x^\alpha,
 \qquad
 x^\alpha=x_1^{\alpha_1}\cdots x_N^{\alpha_N}.
\end{equation}
The group operation in $\mathbb Z_q^N$ appears in the exponents.  More
precisely, if $\alpha,\beta\in\mathbb Z_q^N$, then
\begin{equation}\label{eq:mod-q-exponents-intro}
 x^\alpha x^\beta=x^{\alpha+\beta},
\end{equation}
where the addition $\alpha+\beta$ is taken coordinatewise modulo $q$.
The Fourier transform of $f$ is the map
\[
 \widehat f:\mathbb Z_q^N\longrightarrow\C,
\]
and the Fourier coefficient corresponding to $\alpha$ is
\begin{equation}\label{eq:Fourier-coeff-intro}
 \widehat f(\alpha)
 =\int_{C_q^N}f(x)\overline{x^\alpha}\,d\mu_q^N(x)
 =\frac1{q^N}\sum_{x\in C_q^N}f(x)\overline{x^\alpha}.
\end{equation}
For such a Fourier index, its \emph{total degree} is
\begin{equation}\label{eq:total-degree-intro}
 |\alpha|:=\alpha_1+\cdots+\alpha_N,
\end{equation}
where the fixed representatives $0,\ldots,q-1$ are used.  We denote by
$\BHdeg{d}{q}$ the least constant, independent of $N$, for which
\begin{equation}\label{eq:degree-BH-intro}
 \left(
 \sum_{|\alpha|\leq d}|\widehat f(\alpha)|^{p_d}
 \right)^{1/p_d}
 \leq \BHdeg{d}{q}\,\|f\|_\infty
\end{equation}
{\color{black}holds for every $f:C_q^N\to\C$ satisfying}
\begin{equation}\label{eq:degree-support-intro}
 \widehat f(\alpha)=0\qquad\text{whenever }|\alpha|>d.
\end{equation}
The problem is to understand the dependence of
$\BHdeg{d}{q}$ on $d$ when $q$ is fixed.

\subsection*{Why the finite cyclic case is different}

It is useful to isolate the obstruction before discussing the result.  The
Fourier expansion of $f$ determines the ordinary polynomial
\begin{equation}\label{eq:ordinary-extension-intro}
 \begin{aligned}
 Q_f:\C^N&\longrightarrow\C,\\
 Q_f(z)&:=
 \sum_{\alpha\in\{0,\ldots,q-1\}^N}
 \widehat f(\alpha)z^\alpha,
 \qquad z\in\C^N.
 \end{aligned}
\end{equation}
whose restriction to $C_q^N$ is $f$.  Let
$\operatorname{conv}(C_q)$ denote the convex hull, in the complex plane, of
the $q$ roots of unity.  As explained by Slote, Volberg and Zhang
\cite[Section~2]{SVZ}, a direct adaptation of the classical
Bohnenblust--Hille argument naturally produces an estimate of the form
\begin{equation}\label{eq:convex-hull-obstruction-intro}
 \left(
 \sum_{|\alpha|\leq d}|\widehat f(\alpha)|^{p_d}
 \right)^{1/p_d}
 \leq C(d,q)
 \sup_{z\in\operatorname{conv}(C_q)^N}|Q_f(z)|.
\end{equation}
The difficulty is that the norm on the right-hand side of
\eqref{eq:convex-hull-obstruction-intro} is not the norm in
\eqref{eq:degree-BH-intro}: for $q=2$ the extension is affine in each
coordinate and the supremum over $[-1,1]^N$ reduces to the Boolean cube,
whereas for $q\geq3$ no analogous reduction is available; see
\cite[Section~2]{SVZ}.

Slote, Volberg and Zhang proved dimension-free cyclic
Bohnenblust--Hille inequalities; for prime $q$ they obtained a bound of the
form $K_q^{d^2}$, where $K_q$ depends only on $q$
\cite[Theorem~1]{SVZ}.  Becker, Klein, Slote,
Volberg and Zhang later obtained
\begin{equation}\label{eq:BKSVZ-intro-bound}
 \BHdeg{d}{q}\leq(C\log q)^{2d}
\end{equation}
with a universal constant $C$ \cite[Corollary~7]{BKSVZ}, and asked whether,
for every fixed $3\leq q<\infty$,
\begin{equation}\label{eq:BKSVZ-question-intro}
 \BHdeg{d}{q}=\exp(o(d))
 \qquad(d\to\infty).
\end{equation}
{\color{black}Building on the dimension-free cyclic theory of Slote, Volberg and Zhang,
the quantitative question of Becker, Klein, Slote, Volberg and Zhang, and the
support-sensitive framework of Defant, Galicer, Mansilla, Masty\l o and Muro,
we answer this question affirmatively.  The result is proved in a slightly
larger class, which we now define.}

For a Fourier index $\alpha\in\{0,\ldots,q-1\}^N$, define the support map
\[
 \supp:\{0,\ldots,q-1\}^N\longrightarrow
 \mathcal P(\{1,\ldots,N\})
\]
by
\begin{equation}\label{eq:support-intro}
 \supp\alpha:=\{j\in\{1,\ldots,N\}:\alpha_j\neq0\}.
\end{equation}
The \emph{interaction-order map} is
\[
 s:\{0,\ldots,q-1\}^N\longrightarrow\{0,1,\ldots,N\},
\]
defined by
\begin{equation}\label{eq:interaction-order-intro}
 s(\alpha):=|\supp\alpha|.
\end{equation}
Let $\BHint{d}{q}$ be the least constant,
independent of $N$, such that
\begin{equation}\label{eq:interaction-BH-intro}
 \left(
 \sum_{s(\alpha)\leq d}|\widehat f(\alpha)|^{p_d}
 \right)^{1/p_d}
 \leq \BHint{d}{q}\,\|f\|_\infty
\end{equation}
for every $f:C_q^N\to\C$ satisfying
$\widehat f(\alpha)=0$ whenever $s(\alpha)>d$.  Since
$s(\alpha)\leq|\alpha|$, every function of total degree at most $d$ has
interaction order at most $d$.  Hence
\begin{equation}\label{eq:constant-comparison-intro}
 \BHdeg{d}{q}\leq\BHint{d}{q}.
\end{equation}
For $q\geq3$ this is genuinely a larger class.  For example,
$x_1^{q-1}\cdots x_d^{q-1}$ has interaction order $d$ but total degree
$(q-1)d$.

The interaction-order parameter has appeared in related forms.  For complex
homogeneous polynomials, Carando, Defant and Sevilla-Peris studied monomials
involving only a prescribed number of variables \cite{CDS}; see also
\cite{MNP,CNS}.  Defant, Galicer, Mansilla, Masty\l o and Muro recently
considered support-sensitive levels on Hamming schemes
\cite[Theorem~2.13]{DGMMM}.  {\color{black}In particular, their work already provides dimension-free control with exponential dependence on the interaction order for the class considered here.}  In the notation above, their result implies, for
fixed $q$,
\begin{equation}\label{eq:DGMMM-intro-bound}
 \BHint{d}{q}\leq D_1(q)^d.
\end{equation}
{\color{black}Thus exponential growth in the interaction order was already known.  The point of Theorem~\ref{thm:main} is to replace that exponential dependence, for every fixed $q$, by a subexponential one.}
{\color{black}All asymptotic statements in $d$ are for fixed $q$ unless explicitly stated otherwise.}
Set
\begin{equation}\label{eq:gamma-cq-intro}
 \gamma_2:=\frac12,
 \qquad
 \gamma_q:=\frac{q\log(q-1)}{4(q-2)}\quad(q\geq3),
 \qquad
 c_q:=2\sqrt{2\gamma_q}.
\end{equation}
{\color{black}We can now state our main result.}

\begin{maintheorem}[Subexponential interaction-order BH]\label{thm:main}
For every fixed $q\geq2$, as $d\to\infty$,
\begin{equation}\label{eq:main-theorem-intro}
 \BHint{d}{q}\leq
 \exp\left(
 c_q\sqrt{d\log d}
 +O_q\left(\sqrt{\frac d{\log d}}\log\log d\right)
 \right).
\end{equation}
\end{maintheorem}

\begin{corollary}[Total-degree consequence]\label{cor:total-degree}
For every fixed $q\geq2$, as $d\to\infty$,
\begin{equation}\label{eq:total-degree-consequence-intro}
 \BHdeg{d}{q}\leq
 \exp\left(
 c_q\sqrt{d\log d}
 +O_q\left(\sqrt{\frac d{\log d}}\log\log d\right)
 \right).
\end{equation}
In particular, $\BHdeg{d}{q}=\exp(o(d))$.  For every fixed $q\geq3$, this
{\color{black}answers affirmatively the subexponentiality question posed in \cite[Question~2]{BKSVZ}.}
\end{corollary}

{\color{black}
The dependence of Bohnenblust--Hille constants on the degree is an important
quantitative issue, both in classical harmonic analysis and in more recent
applications to learning theory and quantum information; see, for instance,
\cite{DMP,SVZ,SloteDense}.  Subexponential growth is known in commutative
settings such as complex polynomials on the polytorus and the Boolean cube
\cite{BPSS,DMP}.  In contrast, a recent result shows that the noncommutative
Bohnenblust--Hille constants for qubit systems grow exponentially with the
degree \cite{SloteDense}.  Thus one way to view Theorem~\ref{thm:main} is that
it places the finite cyclic groups $C_q^N$, including the genuinely
non-Boolean cases $q\geq3$, on the subexponential side of this emerging
commutative/noncommutative contrast.
}

\subsection*{Proof strategy}

We now introduce the objects used in the proof and record the estimates that
will be invoked later.

\medskip
\noindent\emph{1. Homogenization.}
Let
\begin{equation}\label{eq:decorated-set-intro}
 D_{q,N}:=\{0\}\cup\bigl(\{1,\ldots,N\}\times\{1,\ldots,q-1\}\bigr).
\end{equation}
The symbol $0$ will be used only to fill empty positions.  If
$\alpha$ has interaction order $r$, then the character $x^\alpha$ is encoded
by the $d$ symbols
\[
 \underbrace{0,\ldots,0}_{d-r\ \mathrm{times}},
 \qquad
 (j,\alpha_j),\quad j\in\supp\alpha.
\]
Since $|D_{q,N}|=1+N(q-1)$, we write the coordinates of
$\C^{1+N(q-1)}$ as
\[
 z=\bigl(z_0,(z_{(j,u)})_{1\leq j\leq N,\ 1\leq u\leq q-1}\bigr).
\]
Define the $d$-homogeneous polynomial
\begin{equation}\label{eq:homogenization-intro}
 \begin{aligned}
 P_f:\C^{1+N(q-1)}&\longrightarrow\C,\\
 P_f(z)&:=
 \sum_{s(\alpha)\leq d}\widehat f(\alpha)
 z_0^{d-s(\alpha)}
 \prod_{j\in\supp\alpha}z_{(j,\alpha_j)}.
 \end{aligned}
\end{equation}
Every monomial in \eqref{eq:homogenization-intro} has exactly $d$ factors.
Define the map
\begin{equation}\label{eq:zeta-intro}
 \begin{aligned}
 \zeta:C_q^N&\longrightarrow\C^{1+N(q-1)},\\
 x&\longmapsto\zeta(x),
 \end{aligned}
 \qquad
 \zeta(x)_0:=1,
 \quad
 \zeta(x)_{(j,u)}:=x_j^u
 \quad(1\leq j\leq N,\ 1\leq u\leq q-1).
\end{equation}
Then
\begin{align}
 P_f(\zeta(x))
 &=\sum_{s(\alpha)\leq d}\widehat f(\alpha)
   1^{d-s(\alpha)}
   \prod_{j\in\supp\alpha}x_j^{\alpha_j}\notag\\
 &=\sum_{s(\alpha)\leq d}\widehat f(\alpha)x^\alpha
 =f(x).
 \label{eq:homogenization-recovery-intro}
\end{align}
Let
\begin{equation}\label{eq:Lf-domain-intro}
 L_f:\bigl(\C^{1+N(q-1)}\bigr)^d\longrightarrow\C
\end{equation}
be the symmetric $d$-linear form associated with $P_f$.  Thus
\begin{equation}\label{eq:Lf-intro}
 L_f(z,\ldots,z)=P_f(z),
 \qquad z\in\C^{1+N(q-1)}.
\end{equation}

\medskip
\noindent\emph{2. The block recurrence.}
The remaining ingredients are collected, with their precise statements, in
Section~\ref{sec:ingredients}.  Their roles are distinct.  Blei's mixed-norm
inequality separates the $d$ positions into an $m$-block and a
$(d-m)$-block.  Potts hypercontractivity controls the second block in
$L_2$ at a cost depending only on its interaction order.  The defining
Bohnenblust--Hille inequality at level $m$ is then applied to the first
block.  Finally, orbit counting keeps track of the passage from canonical
Fourier coefficients to the symmetric multilinear form, while mixed
polarization returns the resulting mixed value of $L_f$ to the original
discrete supremum norm.

Keeping these four losses separate until the last step gives the recurrence
\begin{equation}\label{eq:recurrence-preview}
 \BHint{d}{q}\leq
 \BHint{m}{q}
 \left(\frac{m+1}{m-1}\right)^{\gamma_q(d-m)}
 \binom{2d}{2m},
 \qquad 2\leq m\leq d/2.
\end{equation}
The factor containing $\gamma_q$ is the Potts loss, whereas the binomial
factor is the combined orbit-normalization and polarization loss.  Choosing
\begin{equation}\label{eq:optimal-m-intro}
 m\sim\sqrt{\frac{2\gamma_q d}{\log d}}
\end{equation}
balances these two contributions and yields Theorem~\ref{thm:main}.

Section~\ref{sec:notation} records the remaining Fourier notation.
Section~\ref{sec:ingredients} collects the four ingredients just described,
Section~\ref{sec:recurrence} proves \eqref{eq:recurrence-preview} and derives
the initial bound needed for the asymptotic step,
Section~\ref{sec:proof-main} completes the proof of Theorem~\ref{thm:main},
and Section~\ref{sec:bohr} gives the Bohr-radius application.  The complete
semigroup and log--Sobolev derivation of the Potts estimate is deferred to
Appendix~\ref{app:potts-details}.

\section{Fourier preliminaries}\label{sec:notation}

With the normalized counting measure $\mu_q^N$ from the introduction, the
characters $x\mapsto x^\alpha$ in \eqref{eq:Fourier-intro} form an
orthonormal basis of $L_2(C_q^N)$.  Hence Parseval's identity is
{\color{black}
\begin{equation}\label{eq:parseval-cyclic}
 \|f\|_2^2=\sum_{\alpha\in\mathbb Z_q^N} |\widehat f(\alpha)|^2.
\end{equation}}

For every integer $r\geq1$, set
\begin{equation}\label{eq:pr}
 p_r:=\frac{2r}{r+1}.
\end{equation}
This extends the notation $p_d$ from \eqref{eq:pd-intro} to arbitrary
interaction levels $r$.  For every integer $n\geq1$, we also write
\begin{equation}\label{eq:bracket-n}
 [n]:=\{1,\ldots,n\}.
\end{equation}

For $r\geq0$, define the exact interaction layer
\begin{equation}\label{eq:Vr-cyclic}
 \mathcal V_r(C_q^N)
 :=\operatorname{span}\{x^\alpha:s(\alpha)=r\},
 \qquad
 \mathcal V_{\leq d}(C_q^N)
 :=\bigoplus_{r=0}^d\mathcal V_r(C_q^N),
\end{equation}
with $\mathcal V_r(C_q^N)=\{0\}$ when $r>N$.  Here $\oplus$ denotes a
direct sum: every $g\in\mathcal V_{\leq d}(C_q^N)$ has a unique
decomposition
\[
 g=g_0+\cdots+g_d,
 \qquad g_r\in\mathcal V_r(C_q^N).
\]
The sum is direct because the Fourier supports corresponding to the distinct
conditions $s(\alpha)=r$ are disjoint.  Thus $\mathcal V_r(C_q^N)$ consists
exactly of the functions supported on the interaction layer $r$, while
$\mathcal V_{\leq d}(C_q^N)$ is the class appearing in
\eqref{eq:interaction-BH-intro}.

\section{Ingredients for the recurrence}\label{sec:ingredients}

\subsection{Potts hypercontractivity}\label{sec:potts}

{\color{black}For $0\leq\rho\leq1$ and $1\leq p\leq\infty$, define}
\begin{equation}\label{eq:Potts-operator}
 \begin{aligned}
 T_\rho:L^p(C_q,\mu_q)&\longrightarrow L^p(C_q,\mu_q),\\
 (T_\rho g)(a)&:=\rho g(a)
 +(1-\rho)\int_{C_q}g\,d\mu_q,
 \qquad a\in C_q.
 \end{aligned}
\end{equation}
{\color{black}
For functions on $C_q^N$ we apply the same operator separately in each
coordinate.  More precisely, for $1\leq j\leq N$ let
\[
 \begin{aligned}
 (T_{\rho,j}f)(x_1,\ldots,x_N)
 &: =\rho f(x_1,\ldots,x_N)\\
 &\quad +(1-\rho)\frac1q\sum_{a\in C_q}
 f(x_1,\ldots,x_{j-1},a,x_{j+1},\ldots,x_N),
 \end{aligned}
\]
and define
\begin{equation}\label{eq:Potts-product-domain}
 T_\rho^{\otimes N}:=T_{\rho,1}\cdots T_{\rho,N}.
\end{equation}
The coordinate operators commute, so the order is irrelevant.  Moreover, for
a character $x^\alpha=x_1^{\alpha_1}\cdots x_N^{\alpha_N}$,
\[
 T_{\rho,j}x^\alpha=
 \begin{cases}
 x^\alpha,&\alpha_j=0,\\
 \rho x^\alpha,&\alpha_j\neq0,
 \end{cases}
\]
because $q^{-1}\sum_{a\in C_q}a^u$ equals $1$ for $u=0$ and $0$ for
$1\leq u\leq q-1$.  Hence
\begin{equation}\label{eq:Potts-product-Fourier}
 T_\rho^{\otimes N}x^\alpha=\rho^{s(\alpha)}x^\alpha,
 \qquad
 T_\rho^{\otimes N}f
 =\sum_{\alpha\in\mathbb Z_q^N}
 \rho^{s(\alpha)}\widehat f(\alpha)x^\alpha.
\end{equation}
}
\begin{proposition}[\textcolor{black}{Potts hypercontractivity}]\label{prop:Potts-body}
\textcolor{black}{Let $1<p\leq2$, and recall that $\gamma_q$ is defined in \eqref{eq:gamma-cq-intro}. For every $N\geq1$ and every function}
$f:C_q^N\to\C$,
\begin{equation}\label{eq:Potts-HC}
 \|T_\rho^{\otimes N}f\|_2\leq\|f\|_p
 \qquad\text{whenever}\qquad
 0\leq\rho\leq(p-1)^{\gamma_q}.
\end{equation}
\end{proposition}

{\color{black}
\begin{corollary}[Low-interaction $L_p$--$L_2$ estimate]\label{cor:low-support-Lp-L2}
Let $1<p\leq2$, $k\geq0$, and let $f:C_q^N\to\C$ satisfy
\[
 \widehat f(\alpha)=0\qquad\text{whenever }s(\alpha)>k.
\]
Then
\begin{equation}\label{eq:low-support-Lp-L2}
 \|f\|_2\leq(p-1)^{-\gamma_q k}\|f\|_p.
\end{equation}
\end{corollary}

\begin{proof}
Set $\rho=(p-1)^{\gamma_q}$.  By \eqref{eq:Potts-product-Fourier}, Parseval,
and the assumption $s(\alpha)\leq k$ on the Fourier support of $f$,
\begin{align*}
 \|T_\rho^{\otimes N}f\|_2^2
 &=\sum_{s(\alpha)\leq k}
   \rho^{2s(\alpha)}|\widehat f(\alpha)|^2\\
 &\geq \rho^{2k}
   \sum_{s(\alpha)\leq k}|\widehat f(\alpha)|^2
 =\rho^{2k}\|f\|_2^2.
\end{align*}
Proposition~\ref{prop:Potts-body} gives
$\|T_\rho^{\otimes N}f\|_2\leq\|f\|_p$.  Therefore
$\rho^k\|f\|_2\leq\|f\|_p$, which is \eqref{eq:low-support-Lp-L2}.
\end{proof}
}

{\color{black}For later use in the block recurrence, we isolate the particular form of Corollary~\ref{cor:low-support-Lp-L2} obtained with the Bohnenblust--Hille exponent $p_m$ and interaction order $d-m$.}

{\color{black}
\begin{corollary}[Potts factor for the recurrence]\label{cor:Potts-factor-body}
Let $2\leq m<d$. If $g:C_q^N\to\C$ has interaction order at most $d-m$, then
\begin{equation}\label{eq:Potts-factor-body}
 \|g\|_2
 \leq
 \left(\frac{m+1}{m-1}\right)^{\gamma_q(d-m)}
 \|g\|_{p_m}.
\end{equation}
\end{corollary}

\begin{proof}
Apply Corollary~\ref{cor:low-support-Lp-L2} with
\[
 p=p_m=\frac{2m}{m+1},\qquad k=d-m.
\]
Since
\[
 p_m-1
 =\frac{2m}{m+1}-1
 =\frac{m-1}{m+1},
\]
we obtain
\begin{align*}
 \|g\|_2
 &\leq (p_m-1)^{-\gamma_q(d-m)}\|g\|_{p_m}\\
 &=\left(\frac{m-1}{m+1}\right)^{-\gamma_q(d-m)}\|g\|_{p_m}\\
 &=\left(\frac{m+1}{m-1}\right)^{\gamma_q(d-m)}\|g\|_{p_m}.
\end{align*}
\end{proof}}

\subsection{Blei's mixed-norm inequality}\label{sec:blei}

{\color{black}
The block form of Blei's inequality stated below is the only form needed in
this paper.  It has constant one; see \cite[Theorem~2.1]{BPSS}, and see also
Blei's original inequality \cite{Blei}.  We state it as a lemma for later
reference.

\begin{lemma}[Blei's mixed-norm inequality]\label{lem:blei-block}
Let $I$ be a nonempty finite set and identify $I^d$ with the set of maps
$[d]\to I$.  For $A\subset[d]$, write
\[
 I^A:=\{\mathbf i_A:A\to I\}.
\]
If $S\subset[d]$, $|S|=m$, and
$\mathbf i_S\in I^S$, $\mathbf i_{S^c}\in I^{S^c}$, let
$\mathbf i_S\oplus\mathbf i_{S^c}\in I^{[d]}\cong I^d$ denote the unique
map whose restrictions to $S$ and $S^c$ are $\mathbf i_S$ and
$\mathbf i_{S^c}$, respectively.  For a scalar array
$b=(b_{\mathbf i})_{\mathbf i\in I^d}$, set
\[
 M_S(b):=
 \left[
 \sum_{\mathbf i_S\in I^S}
 \left(
 \sum_{\mathbf i_{S^c}\in I^{S^c}}
 |b_{\mathbf i_S\oplus\mathbf i_{S^c}}|^2
 \right)^{p_m/2}
 \right]^{1/p_m}.
\]
Then, for every $1\le m\le d$,
\begin{equation}\label{eq:blei-body}
 \left(\sum_{\mathbf i\in I^d}|b_{\mathbf i}|^{p_d}\right)^{1/p_d}
 \leq
 \left(
 \prod_{\substack{S\subset[d]\\ |S|=m}}M_S(b)
 \right)^{1/\binom dm}.
\end{equation}
\end{lemma}

The exponents are calibrated by the identity
\[
 \frac md\frac1{p_m}+\frac{d-m}{d}\frac12
 =\frac1{p_d}.
\]
}

\subsection{A polarization estimate}\label{sec:polarization}

{\color{black}We keep the set $D_{q,N}$, the map $\zeta$, the homogeneous
polynomial $P_f$, and the symmetric form $L_f$ from
\eqref{eq:decorated-set-intro}--\eqref{eq:Lf-intro}.  In particular,
$P_f(\zeta(x))=f(x)$ by \eqref{eq:homogenization-recovery-intro}.  When a
vector is repeated in an argument list, we use the shorthand $z^r$ for $r$
repetitions; thus $L_f(z^m,w^{d-m})$ denotes a mixed evaluation of $L_f$.}

\subsubsection{Bernstein coefficients}

{\color{black}The coefficient estimate below is due to Defant, Masty\l o
and P\'erez \cite[Proposition~3.2]{DMP}.  Since its precise constant is used
in the polarization argument, we include a self-contained proof in the
normalization needed here.}

\begin{lemma}[Bernstein coefficient estimate]
\label{lem:Bernstein-coefficient-estimate}
Let
\begin{equation}\label{eq:Bernstein-expansion}
 h(t)=\sum_{r=0}^d b_r\binom dr t^r(1-t)^{d-r},
 \qquad 0\leq t\leq1,
\end{equation}
be a complex polynomial of degree at most $d$.  Then, for
$0\leq m\leq d$,
\begin{equation}\label{eq:Bernstein-coefficient-bound}
 \binom dm |b_m|
 \leq
 \left(
 \sum_{j=0}^{\min\{m,d-m\}}
 4^j\binom d{2j}\binom{d-2j}{m-j}
 \right)
 \|h\|_{C[0,1]}.
\end{equation}
\end{lemma}

{\color{black}
\begin{proof}
Put
\[
 u=\frac{1+s}{2},\qquad v=\frac{1-s}{2},
\]
and define
\[
 H(s):=h\!\left(\frac{1+s}{2}\right),\qquad -1\le s\le1.
\]
Then
\begin{equation}\label{eq:Bernstein-H-basis}
 H(s)=\sum_{r=0}^d a_r u^r v^{d-r},
 \qquad a_r:=\binom dr b_r,
\end{equation}
and
\[
 \|H\|_{C[-1,1]}=\|h\|_{C[0,1]}.
\]
Thus it is enough to estimate the coefficient $a_m$ in the basis
$\{u^r v^{d-r}\}_{r=0}^d$.

We first record the extremal property of the Chebyshev polynomial that is
needed for this coefficient.  Let
\[
 \tau_j:=\cos\frac{j\pi}{d},\qquad 0\le j\le d,
\]
and let $\ell_j$ be the Lagrange polynomial corresponding to these nodes.
If $\Lambda_m(Q)$ denotes the coefficient of $u^m v^{d-m}$ in a polynomial
$Q$ of degree at most $d$, then interpolation gives
\[
 \Lambda_m(Q)=\sum_{j=0}^d Q(\tau_j)\Lambda_m(\ell_j).
\]
We claim that
\begin{equation}\label{eq:Lambda-sign}
 \operatorname{sgn}\Lambda_m(\ell_j)
 =(-1)^{d-m+j}
\end{equation}
whenever $\Lambda_m(\ell_j)\ne0$.  Indeed,
\[
 s-\tau_k=u(1-\tau_k)-v(1+\tau_k),
\]
so the coefficient of $u^m v^{d-m}$ in
$\prod_{k\ne j}(s-\tau_k)$ has sign $(-1)^{d-m}$, since all remaining
factors occurring in that coefficient are nonnegative.  On the other hand,
$\tau_0>\tau_1>\cdots>\tau_d$, and therefore
\[
 \operatorname{sgn}\prod_{k\ne j}(\tau_j-\tau_k)=(-1)^j.
\]
This proves \eqref{eq:Lambda-sign}.

Now $T_d(\tau_j)=(-1)^j$.  Hence \eqref{eq:Lambda-sign} implies
\[
 |\Lambda_m(T_d)|
 =\sum_{j=0}^d |\Lambda_m(\ell_j)|.
\]
Consequently, for every polynomial $Q$ of degree at most $d$,
\begin{align}
 |\Lambda_m(Q)|
 &\le \|Q\|_{C[-1,1]}
       \sum_{j=0}^d |\Lambda_m(\ell_j)| \notag\\
 &=|\Lambda_m(T_d)|\,\|Q\|_{C[-1,1]}.
 \label{eq:Chebyshev-Bernstein-extremal}
\end{align}

It remains to compute $\Lambda_m(T_d)$.  If $u,v\ge0$ and $u+v=1$, write
$u=\cos^2(\theta/2)$ and $v=\sin^2(\theta/2)$.  Since $s=u-v=\cos\theta$,
De Moivre's formula gives
\begin{align*}
 T_d(s)
 &=\cos(d\theta)
 =\operatorname{Re}\bigl(\sqrt u+i\sqrt v\bigr)^{2d}\\
 &=\sum_{r=0}^d(-1)^{d-r}\binom{2d}{2r}u^r v^{d-r}.
\end{align*}
Both sides are polynomials in $s$, so the identity holds identically.  In
particular,
\begin{equation}\label{eq:Chebyshev-Bernstein-coeff}
 |\Lambda_m(T_d)|=\binom{2d}{2m}.
\end{equation}
Finally, taking the coefficient of $x^{2m}$ in
\[
 (1+x)^{2d}=(1+2x+x^2)^d
\]
yields
\begin{equation}\label{eq:binomial-Bernstein-sum}
 \binom{2d}{2m}
 =\sum_{j=0}^{\min\{m,d-m\}}
 4^j\binom d{2j}\binom{d-2j}{m-j}.
\end{equation}
Applying \eqref{eq:Chebyshev-Bernstein-extremal} to $Q=H$, and using
\eqref{eq:Bernstein-H-basis}, \eqref{eq:Chebyshev-Bernstein-coeff}, and
\eqref{eq:binomial-Bernstein-sum}, gives exactly
\eqref{eq:Bernstein-coefficient-bound}.
\end{proof}
}

\subsubsection{Random mixing and mixed polarization}

{\color{black}
The following argument adapts the Bernstein-coefficient polarization scheme to the present cyclic setting.  The point is to realize the mixed evaluation of $L_f$ as a Bernstein coefficient of a one-variable polynomial obtained by independently mixing the coordinates of two points of $C_q^N$.  We record the resulting estimate in the precise form needed below.
}

\begin{lemma}[Mixed polarization]\label{lem:qary-mixed-polarization}
For $0\leq m\leq d$ and $x,y\in C_q^N$,
\begin{equation}\label{eq:qary-mixed-polarization}
 |L_f(\zeta(x)^m,\zeta(y)^{d-m})|
 \leq A_{d,m}\|f\|_\infty,
\end{equation}
where
\begin{equation}\label{eq:Adm-sum-q}
 A_{d,m}
 :=\frac1{\binom dm}
 \sum_{r=0}^{\min\{m,d-m\}}
 4^r\binom d{2r}\binom{d-2r}{m-r}.
\end{equation}
\end{lemma}

\begin{proof}
Fix $x,y\in C_q^N$ and $0\leq t\leq1$.  Let
\[
 \Omega_t:=\{0,1\}^N
\]
carry the product probability measure for which the coordinate variables
are independent and satisfy $\mathbb P(\varepsilon_j=1)=t$.  Define the
random variable
\[
 Z_t:\Omega_t\longrightarrow C_q^N
\]
coordinatewise by
\[
 (Z_t(\varepsilon))_j=
 \begin{cases}
 x_j,&\varepsilon_j=1,\\
 y_j,&\varepsilon_j=0.
 \end{cases}
\]
{\color{black}
For a Fourier index $\alpha=(\alpha_1,\ldots,\alpha_N)$ we have
\[
 Z_t^\alpha=\prod_{j=1}^N (Z_t)_j^{\alpha_j}.
\]
Since the random variables $(Z_t)_1,\ldots,(Z_t)_N$ depend on independent
coordinate variables $\varepsilon_1,\ldots,\varepsilon_N$, they are
independent.  Hence
\begin{align*}
 \E Z_t^\alpha
 &=\E\!\left[\prod_{j=1}^N (Z_t)_j^{\alpha_j}\right]\\
 &=\prod_{j=1}^N \E (Z_t)_j^{\alpha_j}.
\end{align*}
For each $j$, the definition of $Z_t$ gives
\[
 \E (Z_t)_j^{\alpha_j}
 =t x_j^{\alpha_j}+(1-t)y_j^{\alpha_j}.
\]
If $\alpha_j=0$, this factor is $t+(1-t)=1$.  Therefore
\begin{equation}\label{eq:expectation-character}
 \E Z_t^\alpha
 =\prod_{j=1}^N
 \bigl(t x_j^{\alpha_j}+(1-t)y_j^{\alpha_j}\bigr)
 =\prod_{j\in\supp\alpha}
 \bigl(t x_j^{\alpha_j}+(1-t)y_j^{\alpha_j}\bigr).
\end{equation}
Using the Fourier expansion of $f$ and linearity of expectation,
\begin{align*}
 \E f(Z_t)
 &=\sum_\alpha \widehat f(\alpha)\,\E Z_t^\alpha\\
 &=\sum_\alpha \widehat f(\alpha)
   \prod_{j\in\supp\alpha}
   \bigl(t x_j^{\alpha_j}+(1-t)y_j^{\alpha_j}\bigr).
\end{align*}
By the definition of the embedding $\zeta$ and of the homogeneous polynomial
$P_f$, the last expression is exactly
\begin{equation}\label{eq:expectation-Pf}
 \E f(Z_t)=P_f\bigl(t\zeta(x)+(1-t)\zeta(y)\bigr).
\end{equation}
}

Expand the right-hand side through the symmetric form:
\begin{align*}
 &P_f\bigl(t\zeta(x)+(1-t)\zeta(y)\bigr)\\
 &\quad=L_f\bigl(t\zeta(x)+(1-t)\zeta(y),\ldots,
                  t\zeta(x)+(1-t)\zeta(y)\bigr).
\end{align*}
After multilinear expansion, choose the $r$ positions in which the first
summand $t\zeta(x)$ occurs.  There are $\binom dr$ such choices; symmetry of
$L_f$ makes all corresponding values equal.  Therefore
\begin{equation}\label{eq:Bernstein-mixed-evaluation}
 \E f(Z_t)
 =\sum_{r=0}^d\binom drt^r(1-t)^{d-r}
 L_f(\zeta(x)^r,\zeta(y)^{d-r}).
\end{equation}
Since $Z_t$ takes values in $C_q^N$,
\[
 |\E f(Z_t)|\leq\E|f(Z_t)|\leq\|f\|_\infty.
\]
Thus \eqref{eq:Bernstein-mixed-evaluation} is a Bernstein expansion whose
uniform norm is at most $\|f\|_\infty$, and whose $m$-th coefficient is
$L_f(\zeta(x)^m,\zeta(y)^{d-m})$.  Applying
Lemma~\ref{lem:Bernstein-coefficient-estimate} and dividing by $\binom dm$
gives \eqref{eq:qary-mixed-polarization}.
\end{proof}

\subsubsection{The polarization constant}

\begin{lemma}[Exact polarization identity]\label{lem:qary-combinatorial-identity}
For $0\leq m\leq d$,
\begin{equation}\label{eq:qary-combinatorial-identity}
 \sum_{r=0}^{\min\{m,d-m\}}
 4^r\binom d{2r}\binom{d-2r}{m-r}
 =\binom{2d}{2m}.
\end{equation}
Consequently,
\begin{equation}\label{eq:Adm-exact-q}
 A_{d,m}=\frac{\binom{2d}{2m}}{\binom dm}.
\end{equation}
\end{lemma}

\begin{proof}
Partition $2d$ objects into $d$ labelled pairs and count subsets containing
exactly $2m$ objects.  The direct count is $\binom{2d}{2m}$.

For the second count, let $\ell$ be the number of pairs contributing exactly one
object and let $c$ be the number contributing both objects.  Then
\[
 \ell+2c=2m.
\]
Hence $\ell$ is even; write $\ell=2r$, so $c=m-r$.  Choose the $2r$ single pairs,
choose one of the two objects from each of them, and then choose the $m-r$
double pairs from the remaining $d-2r$ pairs.  This gives
\[
 \binom d{2r}2^{2r}\binom{d-2r}{m-r}
 =4^r\binom d{2r}\binom{d-2r}{m-r}.
\]
The constraints $c\geq0$ and $c\leq d-2r$ are equivalent to
$0\leq r\leq\min\{m,d-m\}$.  Summing over all admissible $r$ proves
\eqref{eq:qary-combinatorial-identity}.  Division by $\binom dm$ gives
\eqref{eq:Adm-exact-q}.
\end{proof}

{\color{black}
\subsection{Orbit notation and normalization}\label{sec:orbits}

The bookkeeping in this subsection serves only one purpose.  A Fourier
coefficient occurs once in the canonical Fourier expansion, whereas the
corresponding coefficient of the symmetric $d$-linear form is repeated over
all reorderings of its $d$ indices.  We record precisely the factor relating
these two descriptions.  The construction has three steps: first the indices
and their orbits, then the normalization of the symmetric coefficients, and
finally the two-block form that will be inserted into Blei's inequality.

\subsubsection{Indices and their orbits}

{Fix $f\in\mathcal V_{\leq d}(C_q^N)$ and let $D_{q,N}$ be as in}
\eqref{eq:decorated-set-intro}.  {Choose a total order on $D_{q,N}$,
with $0$ first.  For a finite $A\subset\N$, ordered increasingly, let}
\[
 \Ind(A):=\{\mathbf i:A\to D_{q,N}\}.
\]
{The symmetric group of the finite set $A$ acts on $\Ind(A)$ by
permuting positions.  We write $[\mathbf i]$ for the orbit of $\mathbf i$
under this action.  Let $\Indup(A)$ denote the set of nondecreasing indices.
Sorting the values of an index shows that every orbit contains exactly one
representative in $\Indup(A)$.}

A map is called \emph{block-affine} if its nonzero values
$(j,u)$ have pairwise distinct first coordinates $j$.  A block-affine
representative $\mathbf j\in\Indup([d])$ determines a unique Fourier
index $\alpha$: if $(k,u)$ occurs, set $\alpha_k=u$, and set all other
coordinates equal to zero.  Conversely, $\alpha$ determines the multiset
consisting of $d-s(\alpha)$ zeros and the pairs $(k,\alpha_k)$ with
$k\in\supp\alpha$.  Thus this correspondence is bijective.

Define the coefficient function
\[
 a:\Indup([d])\longrightarrow\C
\]
by
\begin{equation}\label{eq:canonical-a}
 a_{\mathbf j}:=
 \begin{cases}
  \widehat f(\alpha),&\mathbf j\text{ is block-affine and corresponds to }\alpha,\\
  0,&\text{otherwise}.
 \end{cases}
\end{equation}
The bijection above gives
\begin{equation}\label{eq:canonical-norm}
 \|\widehat f\|_{\ell_{p_d}}
 =\left(\sum_{\mathbf j\in\Indup([d])}
 |a_{\mathbf j}|^{p_d}\right)^{1/p_d}.
\end{equation}

We use two extensions to all ordered indices.  The orbit-constant extension is
the map
\[
 \widetilde a:\Ind([d])\longrightarrow\C
\]
defined by
\begin{equation}\label{eq:atilde-def}
 \widetilde a_{\mathbf i}:=a_{\mathbf j}
 \quad\text{when }\mathbf j\in\Indup([d])\cap[\mathbf i],
\end{equation}
and the canonical-support extension is the map
\[
 b:\Ind([d])\longrightarrow\C
\]
defined by
\begin{equation}\label{eq:b-def}
 b_{\mathbf i}:=
 \begin{cases}
  a_{\mathbf i},&\mathbf i\in\Indup([d]),\\
  0,&\mathbf i\notin\Indup([d]).
 \end{cases}
\end{equation}
Thus $b$ contains each Fourier coefficient once, whereas
$\widetilde a$ repeats it on the whole orbit.

\subsubsection{Coefficient normalization}

If a map on a set of size $r$ takes distinct values
$\xi_1,\ldots,\xi_s$ with multiplicities
$n_1,\ldots,n_s$, then the orbit--stabilizer formula gives
\begin{equation}\label{eq:general-orbit-size}
 |[\mathbf i]|=\frac{r!}{n_1!\cdots n_s!}.
\end{equation}
Indeed, there are $r!$ permutations of the positions, and the
$n_1!\cdots n_s!$ permutations among equal entries do not change the map.

\begin{lemma}[Orbit normalization]\label{lem:symmetric-coeff}
For every ordered index $\mathbf i\in\Ind([d])$, the coefficient of
$L_f$ at $\mathbf i$ is
\begin{equation}\label{eq:symmetric-form-coeff}
 \frac{\widetilde a_{\mathbf i}}{|[\mathbf i]|}.
\end{equation}
\end{lemma}

\begin{proof}
Write the symmetric form in coordinates as
\[
 L_f(z^{(1)},\ldots,z^{(d)})
 =\sum_{\mathbf i\in\Ind([d])}
 c_{\mathbf i}\prod_{r=1}^d z^{(r)}_{\mathbf i(r)}.
\]
Symmetry implies $c_{\mathbf i}=c_{\mathbf i'}$ whenever
$\mathbf i'$ belongs to the orbit of $\mathbf i$.  On the diagonal,
\[
 P_f(z)=L_f(z,\ldots,z)
 =\sum_{\mathbf i}c_{\mathbf i}
   \prod_{r=1}^d z_{\mathbf i(r)}.
\]
All $|[\mathbf i]|$ members of one orbit produce the same commutative
monomial.  Its coefficient in $P_f$ is $\widetilde a_{\mathbf i}$.  Hence
\[
 |[\mathbf i]|c_{\mathbf i}=\widetilde a_{\mathbf i},
\]
which is \eqref{eq:symmetric-form-coeff}.
\end{proof}

\subsubsection{The two-block form used later}

Fix $S\subset[d]$ with $|S|=m$.  If
$\mathbf i_S\in\Ind(S)$ and
$\mathbf j_{S^c}\in\Ind(S^c)$, let
$\mathbf i_S\oplus\mathbf j_{S^c}$ be the map on $[d]$ equal to $\mathbf i_S$ on
$S$ and to $\mathbf j_{S^c}$ on $S^c$.  For fixed block indices $\mathbf i_S\in\Ind(S)$ and
$\mathbf j_{S^c}\in\Ind(S^c)$, define the functions
\[
 C_q^N\longrightarrow\C,
 \qquad
 x\longmapsto x^{\mathbf i_S},
 \qquad
 y\longmapsto y^{\mathbf j_{S^c}},
\]
by
\begin{equation}\label{eq:xu-yv}
 x^{\mathbf i_S}:=\prod_{r\in S}\zeta(x)_{\mathbf i_S(r)},
 \qquad
 y^{\mathbf j_{S^c}}:=\prod_{r\in S^c}\zeta(y)_{\mathbf j_{S^c}(r)}.
\end{equation}
Zero entries contribute the factor $\zeta(x)_0=\zeta(y)_0=1$.

Assume that $\widetilde a_{\mathbf i_S\oplus\mathbf j_{S^c}}\neq0$.  Then the global
index is block-affine.  {Hence every nonzero entry is distinct} and
the only repeated value is $0$.  Let $z_{\mathbf i_S}$ and
$z_{\mathbf j_{S^c}}$ be the numbers of zeros in the two blocks.  Formula
\eqref{eq:general-orbit-size} gives
\begin{align}
 |[\mathbf i_S]|&=\frac{m!}{z_{\mathbf i_S}!},\label{eq:orbit-u}\\
 |[\mathbf j_{S^c}]|&=\frac{(d-m)!}{z_{\mathbf j_{S^c}}!},\label{eq:orbit-v}\\
 |[\mathbf i_S\oplus\mathbf j_{S^c}]|
 &=\frac{d!}{(z_{\mathbf i_S}+z_{\mathbf j_{S^c}})!}.
 \label{eq:orbit-uv}
\end{align}
Therefore
\begin{align}
 \frac{|[\mathbf i_S\oplus\mathbf j_{S^c}]|}
 {|[\mathbf i_S]|\,|[\mathbf j_{S^c}]|}
 &=\frac{d!}{(z_{\mathbf i_S}+z_{\mathbf j_{S^c}})!}
   \frac{z_{\mathbf i_S}!}{m!}
   \frac{z_{\mathbf j_{S^c}}!}{(d-m)!}\\
 &=\binom dm
   \frac{z_{\mathbf i_S}!z_{\mathbf j_{S^c}}!}
        {(z_{\mathbf i_S}+z_{\mathbf j_{S^c}})!}\\
 &=\frac{\binom dm}
        {\binom{z_{\mathbf i_S}+z_{\mathbf j_{S^c}}}{z_{\mathbf i_S}}}\\
 &\leq\binom dm.
 \label{eq:orbit-ratio}
\end{align}
When the global coefficient is zero there is nothing to estimate.  Hence we
only use \eqref{eq:orbit-ratio} in the nonzero case.

Because $L_f$ is symmetric, its mixed value depends only on the number of
$x$-arguments.  In the identity below we insert $\zeta(x)$ in the positions
of $S$ and $\zeta(y)$ in the positions of $S^c$.

Define the coefficient function
\[
 C^S:\Indup(S)\times\Indup(S^c)\longrightarrow\C
\]
by
\begin{equation}\label{eq:CuvS}
 C^S_{\mathbf i_S,\mathbf j_{S^c}}
 :=\frac{|[\mathbf i_S]|\,|[\mathbf j_{S^c}]|}
 {|[\mathbf i_S\oplus\mathbf j_{S^c}]|}
 \widetilde a_{\mathbf i_S\oplus\mathbf j_{S^c}}.
\end{equation}
Solving this identity for $\widetilde a$ and using
\eqref{eq:orbit-ratio} gives
\begin{equation}\label{eq:atilde-C-bound}
 |\widetilde a_{\mathbf i_S\oplus\mathbf j_{S^c}}|
 \leq\binom dm|C^S_{\mathbf i_S,\mathbf j_{S^c}}|.
\end{equation}

\begin{lemma}[Mixed evaluation identity]\label{lem:mixed-evaluation}
Fix $S\subset[d]$ with $|S|=m$.  For every $x,y\in C_q^N$,
\begin{equation}\label{eq:mixed-evaluation}
 L_f(\zeta(x)^m,\zeta(y)^{d-m})
 =\sum_{\mathbf i_S\in\Indup(S)}
  \sum_{\mathbf j_{S^c}\in\Indup(S^c)}
 C^S_{\mathbf i_S,\mathbf j_{S^c}}x^{\mathbf i_S}y^{\mathbf j_{S^c}}.
\end{equation}
\end{lemma}

\begin{proof}
Start from the ordered-coordinate expansion of $L_f$.  Fix one global orbit
with nonzero canonical coefficient.  For each pair of block multisets from
this orbit, let $\mathbf i_S\in\Indup(S)$ and
$\mathbf j_{S^c}\in\Indup(S^c)$ be their canonical representatives.  There
are exactly $|[\mathbf i_S]|$ ways to
arrange the entries of the first block in the positions of $S$, and exactly
$|[\mathbf j_{S^c}]|$ ways to arrange the second block in the positions of $S^c$.
Thus exactly
\[
 |[\mathbf i_S]|\,|[\mathbf j_{S^c}]|
\]
ordered members of the global orbit have these two block multisets.
By Lemma~\ref{lem:symmetric-coeff}, each such ordered member has coefficient
\[
 \frac{\widetilde a_{\mathbf i_S\oplus\mathbf j_{S^c}}}
 {|[\mathbf i_S\oplus\mathbf j_{S^c}]|}.
\]
All of them have the same evaluation
$x^{\mathbf i_S}y^{\mathbf j_{S^c}}$ because multiplication is commutative inside
each block.  Their combined contribution is therefore
\[
 \frac{|[\mathbf i_S]|\,|[\mathbf j_{S^c}]|}
 {|[\mathbf i_S\oplus\mathbf j_{S^c}]|}
 \widetilde a_{\mathbf i_S\oplus\mathbf j_{S^c}}
 x^{\mathbf i_S}y^{\mathbf j_{S^c}}
 =C^S_{\mathbf i_S,\mathbf j_{S^c}}x^{\mathbf i_S}y^{\mathbf j_{S^c}}.
\]
Summing over all pairs of block orbits proves
\eqref{eq:mixed-evaluation}.
\end{proof}

\begin{remark}[Normalization check]\label{rem:normalization-audit}
The preceding definitions separate three arrays that should not be
interchanged.  The canonical array \(b\) contains each Fourier coefficient
once.  The orbit-constant array \(\widetilde a\) repeats the same coefficient
on all ordered realizations.  The coefficient of the symmetric form is
\(\widetilde a_{\mathbf i}/|[\mathbf i]|\).  After the split
\([d]=S\sqcup S^c\), summing all ordered realizations with prescribed block
multisets produces precisely
\[
 C^S_{\mathbf i_S,\mathbf j_{S^c}}
 =
 \frac{|[\mathbf i_S]|\,|[\mathbf j_{S^c}]|}
 {|[\mathbf i_S\oplus\mathbf j_{S^c}]|}
 \widetilde a_{\mathbf i_S\oplus\mathbf j_{S^c}}.
\]
Thus the only loss in returning from \(C^S\) to \(\widetilde a\) is
\[
 \frac{|[\mathbf i_S\oplus\mathbf j_{S^c}]|}
 {|[\mathbf i_S]|\,|[\mathbf j_{S^c}]|}
 \leq\binom dm,
\]
which is \eqref{eq:orbit-ratio}.  No additional factorial is suppressed.
The later polarization loss is a different factor,
\(A_{d,m}=\binom{2d}{2m}/\binom dm\).
\end{remark}

}

\section{The recurrence}\label{sec:recurrence}

\subsection{The block estimate}

We now assemble the ingredients from Section~\ref{sec:ingredients}.  Rather
than carrying all of them through a single proof, we separate the argument
into two stages.  The first reduces Blei's mixed-norm inequality to estimates for fixed blocks;
the second estimates each block by orbit normalization, Potts
hypercontractivity, the lower-order Bohnenblust--Hille inequality, and mixed
polarization.

\begin{lemma}[Reduction to block estimates]\label{lem:blei-reduction-recurrence}
Let $2\leq m\leq d/2$, and let
$f\in\mathcal V_{\leq d}(C_q^N)$.  For $S\subset[d]$ with $|S|=m$, set
\begin{equation}\label{eq:QS}
 Q_S:=
 \left[
 \sum_{\mathbf i_S\in\Indup(S)}
 \left(
 \sum_{\mathbf j_{S^c}\in\Indup(S^c)}
 |\widetilde a_{\mathbf i_S\oplus\mathbf j_{S^c}}|^2
 \right)^{p_m/2}
 \right]^{1/p_m}.
\end{equation}
Then
\begin{equation}\label{eq:blei-canonical}
 \|\widehat f\|_{\ell_{p_d}}
 \leq
 \left(
 \prod_{\substack{S\subset[d]\\ |S|=m}}Q_S
 \right)^{1/\binom dm}.
\end{equation}
\end{lemma}

\begin{proof}
Apply Blei's inequality \eqref{eq:blei-body} with $I=D_{q,N}$ to the
canonical-support array $b$ from \eqref{eq:b-def}.  By
\eqref{eq:canonical-norm},
\[
 \left(\sum_{\mathbf i\in\Ind([d])}|b_{\mathbf i}|^{p_d}\right)^{1/p_d}
 =\|\widehat f\|_{\ell_{p_d}}.
\]
If $b_{\mathbf i_S\oplus\mathbf j_{S^c}}\neq0$, then the global index is
nondecreasing.  Its restrictions to the ordered sets $S$ and $S^c$ are
therefore nondecreasing as well.  Thus only canonical block indices occur in
nonzero summands.  Moreover, by \eqref{eq:atilde-def}--\eqref{eq:b-def},
\[
 |b_{\mathbf i_S\oplus\mathbf j_{S^c}}|
 \leq |\widetilde a_{\mathbf i_S\oplus\mathbf j_{S^c}}|.
\]
Replacing the corresponding entries in each Blei factor gives exactly
\eqref{eq:blei-canonical}.
\end{proof}

\begin{lemma}[Estimate of one block]\label{lem:one-block-estimate}
Under the hypotheses of Lemma~\ref{lem:blei-reduction-recurrence}, for every
$S\subset[d]$ with $|S|=m$,
\begin{equation}\label{eq:QS-final}
 Q_S\leq
 \binom dm
 \left(\frac{m+1}{m-1}\right)^{\gamma_q(d-m)}
 \BHint{m}{q}A_{d,m}\|f\|_\infty.
\end{equation}
\end{lemma}

\begin{proof}
Fix $S$.  Recall from \eqref{eq:xu-yv} that $x^{\mathbf i_S}$ denotes the
block character associated with the index $\mathbf i_S$.  All sums below are over
$\mathbf i_S\in\Indup(S)$ and
$\mathbf j_{S^c}\in\Indup(S^c)$.
By the orbit-ratio estimate \eqref{eq:atilde-C-bound},
\begin{equation}\label{eq:QS-orbit}
 Q_S\leq\binom dm
 \left[
 \sum_{\mathbf i_S}
 \left(
 \sum_{\mathbf j_{S^c}}|C^S_{\mathbf i_S,\mathbf j_{S^c}}|^2
 \right)^{p_m/2}
 \right]^{1/p_m}.
\end{equation}

For a fixed $\mathbf i_S$, define
\begin{equation}\label{eq:gu}
 \begin{aligned}
 g_{\mathbf i_S}:C_q^N&\longrightarrow\C,\\
 g_{\mathbf i_S}(y)&:=
 \sum_{\mathbf j_{S^c}}
 C^S_{\mathbf i_S,\mathbf j_{S^c}}y^{\mathbf j_{S^c}}.
 \end{aligned}
\end{equation}
Whenever a coefficient is nonzero, the corresponding canonical block index
is block-affine.  Distinct such indices determine distinct Fourier
characters.  Hence Parseval gives
\begin{equation}\label{eq:Parseval-gu}
 \|g_{\mathbf i_S}\|_2^2
 =\sum_{\mathbf j_{S^c}}|C^S_{\mathbf i_S,\mathbf j_{S^c}}|^2.
\end{equation}
Every character in $g_{\mathbf i_S}$ has interaction order at most $d-m$.
Corollary~\ref{cor:Potts-factor-body} therefore yields
\begin{equation}\label{eq:Potts-gu}
 \left(\sum_{\mathbf j_{S^c}}|C^S_{\mathbf i_S,\mathbf j_{S^c}}|^2\right)^{1/2}
 \leq
 \left(\frac{m+1}{m-1}\right)^{\gamma_q(d-m)}
 \|g_{\mathbf i_S}\|_{p_m}.
\end{equation}
Raising to the power $p_m$, summing in $\mathbf i_S$, and using Fubini on
the finite probability space gives
\begin{equation}\label{eq:Fubini-step}
\begin{aligned}
 &\left[
 \sum_{\mathbf i_S}
 \left(\sum_{\mathbf j_{S^c}}|C^S_{\mathbf i_S,\mathbf j_{S^c}}|^2\right)^{p_m/2}
 \right]^{1/p_m}\\
 &\qquad\leq
 \left(\frac{m+1}{m-1}\right)^{\gamma_q(d-m)}
 \sup_{y\in C_q^N}
 \left[
 \sum_{\mathbf i_S}|g_{\mathbf i_S}(y)|^{p_m}
 \right]^{1/p_m}.
\end{aligned}
\end{equation}

Fix now $y\in C_q^N$ and define
\begin{equation}\label{eq:hy}
 \begin{aligned}
 h_y:C_q^N&\longrightarrow\C,\\
 h_y(x)&:=L_f(\zeta(x)^m,\zeta(y)^{d-m}).
 \end{aligned}
\end{equation}
Lemma~\ref{lem:mixed-evaluation} and the definition of
$g_{\mathbf i_S}$ give
\[
 h_y(x)=\sum_{\mathbf i_S}g_{\mathbf i_S}(y)x^{\mathbf i_S}.
\]
Again the nonzero characters are distinct, and each has interaction order at
most $m$.  Thus, by the definition of $\BHint{m}{q}$,
\begin{equation}\label{eq:BH-hy}
 \left[
 \sum_{\mathbf i_S}|g_{\mathbf i_S}(y)|^{p_m}
 \right]^{1/p_m}
 \leq\BHint{m}{q}\|h_y\|_\infty.
\end{equation}
Finally, Lemma~\ref{lem:qary-mixed-polarization} gives
\begin{equation}\label{eq:hy-polarization}
 \sup_{y\in C_q^N}\|h_y\|_\infty
 \leq A_{d,m}\|f\|_\infty.
\end{equation}
Combining \eqref{eq:QS-orbit}, \eqref{eq:Fubini-step},
\eqref{eq:BH-hy}, and \eqref{eq:hy-polarization} proves
\eqref{eq:QS-final}.
\end{proof}

\subsection{The exact recurrence}

\begin{proposition}[Exact recurrence]\label{prop:exact-recursion}
For $2\leq m\leq d/2$,
\begin{equation}\label{eq:exact-recursion}
 \BHint{d}{q}\leq
 \BHint{m}{q}
 \left(\frac{m+1}{m-1}\right)^{\gamma_q(d-m)}
 \binom{2d}{2m}.
\end{equation}
\end{proposition}

\begin{proof}
By Lemma~\ref{lem:blei-reduction-recurrence} and
Lemma~\ref{lem:one-block-estimate},
\[
 \|\widehat f\|_{\ell_{p_d}}
 \leq
 \binom dm
 \left(\frac{m+1}{m-1}\right)^{\gamma_q(d-m)}
 \BHint{m}{q}A_{d,m}\|f\|_\infty,
\]
since the same bound holds for every Blei factor $Q_S$.  The exact
polarization identity \eqref{eq:Adm-exact-q} gives
\[
 \binom dm A_{d,m}=\binom{2d}{2m}.
\]
Taking the supremum over $N$ and over nonzero
$f\in\mathcal V_{\leq d}(C_q^N)$ proves \eqref{eq:exact-recursion}.
\end{proof}

\subsection{Initial control for the iteration}\label{subsec:coarse-start}

Exponential estimates in the interaction order are already available in the
literature; see the bound recalled in \eqref{eq:DGMMM-intro-bound}.  For the
asymptotic argument below we only need
$\log\BHint{r}{q}=O_q(r)$.  We derive this weaker statement directly in three
short steps.

\begin{lemma}[Interaction order one]\label{lem:interaction-one}
For every fixed $q\geq2$,
\begin{equation}\label{eq:B1-elementary}
 \BHint{1}{q}\leq 1+2\pi(q-1).
\end{equation}
\end{lemma}

\begin{proof}
Let $f\in\mathcal V_{\leq1}(C_q^N)$.  For $j\in[N]$, let
$e_j\in\mathbb Z_q^N$ denote the $j$th coordinate vector and write
\[
 f(x)=c+\sum_{j=1}^Ng_j(x_j),
 \qquad c:=\widehat f(0),
\]
where
\[
 g_j(a):=\sum_{u=1}^{q-1}\widehat f(u e_j)a^u,
 \qquad a\in C_q.
\]
Each $g_j$ has mean zero.  Choose $a_j\in C_q$ with
$|g_j(a_j)|=\|g_j\|_\infty$, and for $\theta\in[0,2\pi]$ set
\[
 \Psi_j(\theta):=
 \max_{a\in C_q}\operatorname{Re}(e^{-i\theta}g_j(a)).
\]
The mean-zero property implies $\Psi_j(\theta)\geq0$.  Writing
$g_j(a_j)=\|g_j\|_\infty e^{i\varphi_j}$ when $g_j(a_j)\neq0$, we obtain
\[
 \Psi_j(\theta)\geq
 \|g_j\|_\infty\max\{0,\cos(\theta-\varphi_j)\}.
\]
Since
\[
 \int_0^{2\pi}\max\{0,\cos\theta\}\,d\theta=2,
\]
we have
\[
 \frac1{2\pi}\int_0^{2\pi}\Psi_j(\theta)\,d\theta
 \geq\frac1\pi\|g_j\|_\infty.
\]
After summing in $j$, some common angle $\theta_0$ satisfies
\[
 \sum_{j=1}^N\Psi_j(\theta_0)
 \geq\frac1\pi\sum_{j=1}^N\|g_j\|_\infty.
\]
Choosing $x_j\in C_q$ realizing the maximum in
$\Psi_j(\theta_0)$ gives
\[
 \frac1\pi\sum_{j=1}^N\|g_j\|_\infty
 \leq
 \left|\sum_{j=1}^Ng_j(x_j)\right|
 \leq\|f-c\|_\infty
 \leq2\|f\|_\infty,
\]
because $c$ is the average of $f$.  Finally,
$|\widehat f(u e_j)|\leq\|g_j\|_\infty$, so
\[
 |\widehat f(0)|+
 \sum_{j=1}^N\sum_{u=1}^{q-1}|\widehat f(u e_j)|
 \leq
 \bigl(1+2\pi(q-1)\bigr)\|f\|_\infty.
\]
This is \eqref{eq:B1-elementary}.
\end{proof}

\begin{lemma}[Endpoint $L_1$--$L_2$ estimate]\label{lem:endpoint-L1-L2}
If $g:C_q^N\to\C$ has interaction order at most $k$, then
\begin{equation}\label{eq:L1-L2-coarse}
 \|g\|_2\leq2^{3\gamma_q k}\|g\|_1.
\end{equation}
\end{lemma}

\begin{proof}
Corollary~\ref{cor:low-support-Lp-L2} with $p=3/2$ gives
\[
 \|g\|_2\leq2^{\gamma_q k}\|g\|_{3/2}.
\]
On the probability space $(C_q^N,\mu_q^N)$,
\[
 \|g\|_{3/2}\leq\|g\|_1^{1/3}\|g\|_2^{2/3}.
\]
For $g\neq0$, divide by $\|g\|_2^{2/3}$ and cube.  The zero function is
trivial.
\end{proof}

\begin{proposition}[Coarse exponential control]\label{prop:coarse-bound}
For every fixed $q\geq2$ and every integer $r\geq2$,
\begin{equation}\label{eq:coarse-exponential}
 \BHint{r}{q}
 \leq
 \BHint{1}{q}\,2^{3\gamma_q(r-1)}\binom{2r}{2}.
\end{equation}
Consequently,
\begin{equation}\label{eq:coarse-bound-log}
 \log\BHint{r}{q}=O_q(r).
\end{equation}
\end{proposition}

\begin{proof}
We use a singleton block directly; this is the endpoint version of the
coefficient decomposition used above, rather than an application of
Lemmas~\ref{lem:blei-reduction-recurrence}--\ref{lem:one-block-estimate}, whose
hypotheses require $m\geq2$.

Fix $S\subset[r]$ with $|S|=1$.  We use the notation of the orbit
construction in Section~\ref{sec:orbits} with $d$ replaced by $r$ and
$m=1$; in particular, $C^S$ denotes the corresponding two-block coefficient
array.  Blei's inequality
\eqref{eq:blei-body}, now with $m=1$ and hence $p_1=1$, gives
\begin{equation}\label{eq:coarse-Blei-singleton}
 \|\widehat f\|_{\ell_{p_r}}
 \leq
 \left(
 \prod_{|S|=1} Q_S^{(1)}
 \right)^{1/r},
\end{equation}
where
\[
 Q_S^{(1)}:=
 \sum_{\mathbf i_S\in\Indup(S)}
 \left(
 \sum_{\mathbf j_{S^c}\in\Indup(S^c)}
 |\widetilde a_{\mathbf i_S\oplus\mathbf j_{S^c}}|^2
 \right)^{1/2}.
\]
As in the proof of Lemma~\ref{lem:one-block-estimate}, the orbit-ratio bound
\eqref{eq:atilde-C-bound} gives
\begin{equation}\label{eq:coarse-orbit-singleton}
 Q_S^{(1)}
 \leq r
 \sum_{\mathbf i_S}
 \left(
 \sum_{\mathbf j_{S^c}}
 |C^S_{\mathbf i_S,\mathbf j_{S^c}}|^2
 \right)^{1/2}.
\end{equation}
For fixed $\mathbf i_S$, set
\[
 g_{\mathbf i_S}(y)
 :=\sum_{\mathbf j_{S^c}}
 C^S_{\mathbf i_S,\mathbf j_{S^c}}y^{\mathbf j_{S^c}}.
\]
The nonzero terms are distinct characters, so Parseval and
Lemma~\ref{lem:endpoint-L1-L2} yield
\[
 \left(
 \sum_{\mathbf j_{S^c}}
 |C^S_{\mathbf i_S,\mathbf j_{S^c}}|^2
 \right)^{1/2}
 =\|g_{\mathbf i_S}\|_2
 \leq
 2^{3\gamma_q(r-1)}\|g_{\mathbf i_S}\|_1.
\]
Summing in $\mathbf i_S$ and using Fubini on the normalized counting measure,
\begin{equation}\label{eq:coarse-Fubini-singleton}
 \sum_{\mathbf i_S}
 \left(
 \sum_{\mathbf j_{S^c}}
 |C^S_{\mathbf i_S,\mathbf j_{S^c}}|^2
 \right)^{1/2}
 \leq
 2^{3\gamma_q(r-1)}
 \sup_{y\in C_q^N}\sum_{\mathbf i_S}|g_{\mathbf i_S}(y)|.
\end{equation}
For fixed $y$, define
\[
 h_y(x):=L_f(\zeta(x),\zeta(y)^{r-1}).
\]
Lemma~\ref{lem:mixed-evaluation} gives
\[
 h_y(x)=\sum_{\mathbf i_S}g_{\mathbf i_S}(y)x^{\mathbf i_S}.
\]
The nonzero characters in this sum are distinct and have interaction order at
most one.  Since $p_1=1$, the definition of $\BHint{1}{q}$ therefore gives
\[
 \sum_{\mathbf i_S}|g_{\mathbf i_S}(y)|
 \leq \BHint{1}{q}\,\|h_y\|_\infty.
\]
Finally, Lemma~\ref{lem:qary-mixed-polarization} with $m=1$ yields
\[
 \sup_y\|h_y\|_\infty\leq A_{r,1}\|f\|_\infty.
\]
Combining this with \eqref{eq:coarse-orbit-singleton} and
\eqref{eq:coarse-Fubini-singleton}, we obtain, for every singleton $S$,
\[
 Q_S^{(1)}
 \leq
 r\,2^{3\gamma_q(r-1)}\BHint{1}{q}A_{r,1}\|f\|_\infty.
\]
Insert this bound into \eqref{eq:coarse-Blei-singleton}.  Since the right-hand
side is independent of $S$,
\[
 \|\widehat f\|_{\ell_{p_r}}
 \leq
 r\,2^{3\gamma_q(r-1)}\BHint{1}{q}A_{r,1}\|f\|_\infty.
\]
Taking the supremum over $N$ and nonzero
$f\in\mathcal V_{\leq r}(C_q^N)$, and using
\[
 rA_{r,1}=\binom{2r}{2}
\]
from \eqref{eq:Adm-exact-q}, proves \eqref{eq:coarse-exponential}.  Taking
logarithms and using Lemma~\ref{lem:interaction-one} gives
\eqref{eq:coarse-bound-log}.
\end{proof}

\subsection{The asymptotic balance}\label{subsec:asymptotic-balance}

The final step uses only the exact recurrence and the coarse estimate above.
We isolate the elementary asymptotic calculation so that the proof of
Theorem~A contains no additional bookkeeping.

\begin{lemma}[Balanced splitting]\label{lem:balanced-splitting}
Fix $q\geq2$ and set
\begin{equation}\label{eq:xd}
 x_d:=\sqrt{\frac{2\gamma_qd}{\log d}},
 \qquad
 m_d:=\lfloor x_d\rfloor.
\end{equation}
Then $m_d\to\infty$, $m_d=o(d)$, and, as $d\to\infty$,
\begin{align}
 \gamma_q(d-m_d)\log\frac{m_d+1}{m_d-1}
 &\leq
 \sqrt{2\gamma_q}\sqrt{d\log d}+O_q(\log d),
 \label{eq:qary-HC-asymptotic}\\
 \log\binom{2d}{2m_d}
 &=
 \sqrt{2\gamma_q}\sqrt{d\log d}
 +O_q\!\left(\sqrt{\frac d{\log d}}\log\log d\right).
 \label{eq:qary-binomial-asymptotic}
\end{align}
\end{lemma}

\begin{proof}
Since $m_d=x_d+O(1)$,
\[
 \frac1{m_d-1}=\frac1{x_d}+O_q\!\left(\frac1{x_d^2}\right).
\]
Using $\log(1+t)\leq t$,
\[
 \gamma_q(d-m_d)\log\frac{m_d+1}{m_d-1}
 \leq\frac{2\gamma_q(d-m_d)}{m_d-1}.
\]
Now $x_d^2=2\gamma_qd/\log d$, and therefore
\[
 \frac{2\gamma_qd}{x_d}
 =\sqrt{2\gamma_q}\sqrt{d\log d},
 \qquad
 \frac d{x_d^2}=O_q(\log d),
\]
which proves \eqref{eq:qary-HC-asymptotic}.

For the second estimate, $m_d=o(d)$ and Stirling's formula gives
\[
 \log\binom{2d}{2m_d}
 =2m_d\log\frac d{m_d}+O(m_d).
\]
At the choice \eqref{eq:xd},
\[
 \log\frac d{m_d}
 =\frac12\log d+O_q(\log\log d),
\]
and substitution gives \eqref{eq:qary-binomial-asymptotic}.
\end{proof}

All quantitative estimates needed for Theorem~A are now in place.

\section{Proof of Theorem A}\label{sec:proof-main}

\begin{proof}[Proof of Theorem~A]
Fix $q$ and choose $m=m_d$ as in Lemma~\ref{lem:balanced-splitting}.  For all
sufficiently large $d$, $2\leq m\leq d/2$, so Proposition~\ref{prop:exact-recursion}
applies.  Taking logarithms gives
\[
 \log\BHint{d}{q}
 \leq
 \log\BHint{m}{q}
 +\gamma_q(d-m)\log\frac{m+1}{m-1}
 +\log\binom{2d}{2m}.
\]
By Proposition~\ref{prop:coarse-bound},
\[
 \log\BHint{m}{q}=O_q(m)
 =O_q\!\left(\sqrt{\frac d{\log d}}\right).
\]
Together with Lemma~\ref{lem:balanced-splitting}, this yields
\[
 \log\BHint{d}{q}
 \leq
 2\sqrt{2\gamma_q}\sqrt{d\log d}
 +O_q\!\left(\sqrt{\frac d{\log d}}\log\log d\right).
\]
Since $c_q=2\sqrt{2\gamma_q}$, Theorem~\ref{thm:main} follows.
\end{proof}

{\color{black}
\begin{remark}[A second-order refinement]\label{rem:second-order-refinement}
The same recurrence contains slightly more quantitative information than is
needed for Theorem~A.  If one keeps the next term in the Stirling expansion
of \(\log\binom{2d}{2m}\) at the choice
\(m\sim\sqrt{2\gamma_qd/\log d}\), then
\[
 \log\BHint{d}{q}
 \leq c_q\sqrt{d\log d}
 +\frac{c_q}{2}\sqrt{\frac d{\log d}}\log\log d
 +O_q\!\left(\sqrt{\frac d{\log d}}\right).
\]
No additional ingredient is required; this is only a more precise asymptotic
reading of the recurrence.  We keep Theorem~A in the simpler form because the
leading subexponential scale is the main point.
\end{remark}
}

\section{A Bohr-radius consequence}\label{sec:bohr}

{\color{black}
We finish with a Bohr-radius consequence of Theorem~A for exact interaction
layers.  This application belongs to a line of work developed by several
authors.  The lower estimate uses the classical Bohnenblust--Hille--H\"older
argument that is standard in multidimensional Bohr-radius problems; see, for
instance, \cite{BPSS}.  In the finite-group setting, Slote, Volberg and Zhang
introduced and studied Bohr-type consequences of cyclic Bohnenblust--Hille
inequalities; see \cite[Section~4.2]{SVZ}.  Defant, Galicer, Mansilla,
Masty\l o and Muro subsequently developed support-sensitive and spherical
Sidon estimates on Hamming schemes, obtaining the corresponding scale up to
factors exponential in the interaction level
\cite[Theorem~3.3]{DGMMM}.  In the Boolean case, Defant, Masty\l o and
P\'erez obtained precise homogeneous Boolean-radius estimates
\cite[Theorem~3.1]{DMPBohr}.

Our point here is to combine these ideas with the subexponential estimate of
Theorem~A.  Since $(\BHint{d}{q})^{1/d}=1+o(1)$, the Bohnenblust--Hille
constant disappears at the $d$th-root scale.  Together with the exact
cardinality of an interaction layer and a standard probabilistic random-sign
argument, this gives matching $d$th-root asymptotics for the interaction
Bohr radius.  Thus the argument should be viewed as a refinement, in the
present cyclic interaction-order setting, of the Bohr/Sidon framework already
present in the works cited above, rather than as an unrelated construction.}

For $1\leq d\leq N$, let
\[
 \mathcal H_{N,q}^{(d)}
 :=\left\{f:C_q^N\to\C:
 \widehat f(\alpha)=0\ \text{whenever }|\supp(\alpha)|\neq d\right\}.
\]
Define the interaction Bohr radius of the $d$th layer by
\begin{equation}\label{eq:bohr-radius-definition}
 R_{N,q}^{(d)}
 :=\sup\left\{0\leq r\leq1:
 r^d\sum_{|\supp(\alpha)|=d}|\widehat f(\alpha)|
 \leq\|f\|_\infty
 \ \text{for every }f\in\mathcal H_{N,q}^{(d)}\right\}.
\end{equation}

Equivalently, let the Sidon constant of the exact interaction layer be
\begin{equation}\label{eq:layer-sidon-definition}
 S_{N,q}^{(d)}
 :=\sup_{0\neq f\in\mathcal H_{N,q}^{(d)}}
 \frac{\displaystyle\sum_{|\supp(\alpha)|=d}|\widehat f(\alpha)|}
      {\|f\|_\infty}.
\end{equation}
Since a single character belongs to \(\mathcal H_{N,q}^{(d)}\) and has
coefficient sum and supremum norm both equal to one, \(S_{N,q}^{(d)}\geq1\).
Thus \eqref{eq:bohr-radius-definition} is exactly the radial form of
\eqref{eq:layer-sidon-definition}, and
\begin{equation}\label{eq:bohr-sidon-identity}
 \boxed{R_{N,q}^{(d)}=\bigl(S_{N,q}^{(d)}\bigr)^{-1/d}.}
\end{equation}
In particular, the terminology ``Bohr radius'' is intrinsic to the
interaction grading: the multiplier $r^{s(\alpha)}$ acts as the scalar $r^d$
on the entire $d$th interaction layer.

{\color{black}The upper estimate uses the classical probabilistic sign method (of Kahane--Salem--Zygmund type); for completeness we give the short finite-group argument, based only on Hoeffding's inequality and a union bound.}
\begin{lemma}[Layerwise random-sign estimate]\label{lem:layer-KSZ}
Let
\[
 \mathcal N_{N,q}^{(d)}:=\binom Nd(q-1)^d.
\]
There are signs $\varepsilon_\alpha\in\{-1,1\}$, indexed by
$|\supp(\alpha)|=d$, such that
\begin{equation}\label{eq:layer-KSZ}
 \left\|
 \sum_{|\supp(\alpha)|=d}\varepsilon_\alpha x^\alpha
 \right\|_\infty
 \leq 2\sqrt{\mathcal N_{N,q}^{(d)}\bigl(N\log q+\log 8\bigr)}.
\end{equation}
\end{lemma}

\begin{proof}
Choose the signs independently and symmetrically.  For a fixed choice of signs, define
\[
 F_\varepsilon:C_q^N\longrightarrow\C,
 \qquad
 F_\varepsilon(x)
 :=\sum_{|\supp(\alpha)|=d}\varepsilon_\alpha x^\alpha.
\]
The real and imaginary parts are sums of independent mean-zero random
variables, each bounded in absolute value by $1$.  Hoeffding's inequality \cite{Hoeffding}
therefore gives, for $t>0$,
\[
 \mathbb P\bigl(|F_\varepsilon(x)|>t\bigr)
 \leq
 4\exp\left(-\frac{t^2}{4\mathcal N_{N,q}^{(d)}}\right).
\]
There are exactly $q^N$ points in $C_q^N$.  Hence the union bound yields
\[
 \mathbb P\bigl(\|F_\varepsilon\|_\infty>t\bigr)
 \leq
 4q^N\exp\left(-\frac{t^2}{4\mathcal N_{N,q}^{(d)}}\right).
\]
With
\[
 t=2\sqrt{\mathcal N_{N,q}^{(d)}\bigl(N\log q+\log 8\bigr)},
\]
the right-hand side is at most $1/2$.  Thus at least one choice of signs
satisfies \eqref{eq:layer-KSZ}.
\end{proof}

{\color{black}We emphasize that the random-sign construction itself is standard; the new
input in the asymptotic comparison below is the subexponential control from
Theorem~A in the cyclic interaction-order setting.}

\begin{theorem}[Interaction Bohr radius]\label{thm:bohr-radius}
For every $q\geq2$ and $1\leq d\leq N$,
\begin{equation}\label{eq:bohr-two-sided}
 {\BHint{d}{q}}^{-1/d}
 \binom Nd^{-\frac{d-1}{2d^2}}
 (q-1)^{-\frac{d-1}{2d}}
 \leq R_{N,q}^{(d)}
 \leq
 \left(2\sqrt{N\log q+\log 8}\right)^{1/d}
 \binom Nd^{-\frac1{2d}}
 (q-1)^{-1/2}.
\end{equation}
Moreover, for fixed $q$, if
\begin{equation}\label{eq:bohr-general-regime}
 d=d(N)\to\infty,
 \qquad d=o(N),
\end{equation}
then
\begin{equation}\label{eq:bohr-general-asymptotic}
 \boxed{
 R_{N,q}^{(d)}
 =(1+o(1))\,N^{1/(2d)}
 \sqrt{\frac{d}{e(q-1)N}}.}
\end{equation}
\end{theorem}

\begin{proof}
There are exactly
\[
 \mathcal N_{N,q}^{(d)}=\binom Nd(q-1)^d
\]
characters of interaction order $d$.  Set $p_d=2d/(d+1)$.  H\"older's
inequality and the definition of $\BHint{d}{q}$ give, for
$f\in\mathcal H_{N,q}^{(d)}$,
\begin{align*}
 \sum_{|\supp(\alpha)|=d}|\widehat f(\alpha)|
 &\leq \bigl(\mathcal N_{N,q}^{(d)}\bigr)^{\frac{d-1}{2d}}
 \left(\sum_{|\supp(\alpha)|=d}
 |\widehat f(\alpha)|^{p_d}\right)^{1/p_d}\\
 &\leq
 \BHint{d}{q}\,
 \bigl(\mathcal N_{N,q}^{(d)}\bigr)^{\frac{d-1}{2d}}\|f\|_\infty.
\end{align*}
Hence every $r$ such that
\[
 r^d\BHint{d}{q}\bigl(\mathcal N_{N,q}^{(d)}\bigr)^{\frac{d-1}{2d}}\leq1
\]
is admissible in \eqref{eq:bohr-radius-definition}.  Taking $d$th roots
gives the lower estimate in \eqref{eq:bohr-two-sided}.

For the reverse estimate, choose $F$ as in Lemma~\ref{lem:layer-KSZ}.  All its
$\mathcal N_{N,q}^{(d)}$ Fourier coefficients have modulus one.  If $r$ is admissible
in \eqref{eq:bohr-radius-definition}, then
\[
 r^d\mathcal N_{N,q}^{(d)}
 \leq\|F\|_\infty
 \leq2\sqrt{\mathcal N_{N,q}^{(d)}\bigl(N\log q+\log8\bigr)}.
\]
Therefore
\[
 r\leq
 \left(2\sqrt{N\log q+\log8}\right)^{1/d}
 \bigl(\mathcal N_{N,q}^{(d)}\bigr)^{-1/(2d)},
\]
which is the upper estimate in \eqref{eq:bohr-two-sided}.

We compare the two bounds under \eqref{eq:bohr-general-regime}.  Write
$L_{N,d,q}$ and $U_{N,d,q}$ for the left- and right-hand sides of
\eqref{eq:bohr-two-sided}.  Their quotient has the exact form
\begin{equation}\label{eq:bohr-ratio-bounds}
 \frac{U_{N,d,q}}{L_{N,d,q}}
 =
 \frac{\bigl(2\sqrt{N\log q+\log8}\,\BHint{d}{q}\bigr)^{1/d}}
 {\bigl(\mathcal N_{N,q}^{(d)}\bigr)^{1/(2d^2)}}.
\end{equation}
Theorem~A gives
\[
 \frac{\log\BHint{d}{q}}{d}=o(1).
\]
Since $q$ is fixed,
\[
 \frac1d\log\left(2\sqrt{N\log q+\log8}\right)
 =\frac{\log N}{2d}+o(1).
\]
Moreover, $d=o(N)$ and Stirling's formula give
\begin{equation}\label{eq:binomial-log-asymptotic-bohr}
 \frac1d\log\binom Nd
 =\log\frac Nd+1+o(1).
\end{equation}
Consequently,
\begin{align*}
 \frac1{2d^2}\log \mathcal N_{N,q}^{(d)}
 &=\frac1{2d}\left(
   \log\frac Nd+1+\log(q-1)+o(1)\right).
\end{align*}
Taking logarithms in \eqref{eq:bohr-ratio-bounds}, the terms involving
$\log N/(2d)$ cancel, and we obtain
\[
 \log\frac{U_{N,d,q}}{L_{N,d,q}}
 =\frac{\log d-1-\log(q-1)}{2d}+o(1)=o(1).
\]
Thus $U_{N,d,q}/L_{N,d,q}\to1$.

Finally, \eqref{eq:binomial-log-asymptotic-bohr} implies
\[
 \binom Nd^{-1/(2d)}
 =(1+o(1))\sqrt{\frac{d}{eN}},
\]
while
\[
 \left(2\sqrt{N\log q+\log8}\right)^{1/d}
 =(1+o(1))N^{1/(2d)}.
\]
Substituting these relations into the upper bound in
\eqref{eq:bohr-two-sided} proves \eqref{eq:bohr-general-asymptotic}.
\end{proof}

Thus the role of Theorem~A in this application is precisely at the
$d$th-root scale: its subexponential constant becomes asymptotically invisible,
while the layer cardinality and the random-sign estimate retain the leading
$q$- and dimension-dependent terms.

The general formula \eqref{eq:bohr-general-asymptotic} contains several
joint growth regimes.  Two contrasting specializations are particularly
transparent: in one, the interaction order lies well above the logarithmic
scale, while in the other it lies precisely on that scale.  We record them
together to make the distinction visible.

\begin{corollary}[Two contrasting Bohr regimes]\label{cor:bohr-complementary-regimes}
Fix $q\geq2$ and suppose $d=d(N)\to\infty$ with $d=o(N)$.  Then the following
two specializations of Theorem~\ref{thm:bohr-radius} hold:
\[
\begin{array}{c@{\qquad\qquad}c}
\boxed{\ \log N=o(d)\ }
&
\boxed{\ d\sim\tau\log N,\quad 0<\tau<\infty\ }\\[2mm]
\displaystyle
R_{N,q}^{(d)}
=(1+o(1))\sqrt{\frac{d}{e(q-1)N}}
&
\displaystyle
R_{N,q}^{(d)}
=(1+o(1))e^{1/(2\tau)}
\sqrt{\frac{d}{e(q-1)N}}.
\end{array}
\]
In particular, at the critical scale $d\sim\log N$,
\begin{equation}\label{eq:bohr-critical-scale}
 R_{N,q}^{(d)}
 =(1+o(1))\sqrt{\frac{\log N}{(q-1)N}}.
\end{equation}
\end{corollary}

\begin{proof}
Both statements follow directly from \eqref{eq:bohr-general-asymptotic}.
If $\log N=o(d)$, then $N^{1/(2d)}\to1$.  If
$d\sim\tau\log N$, then
\[
 N^{1/(2d)}
 =\exp\left(\frac{\log N}{2d}\right)
 \longrightarrow e^{1/(2\tau)}.
\]
For $\tau=1$, this factor cancels the $e^{-1/2}$ appearing in
\eqref{eq:bohr-general-asymptotic}, which gives
\eqref{eq:bohr-critical-scale}.
\end{proof}

\begin{remark}
The two cases in Corollary~\ref{cor:bohr-complementary-regimes} are not separate
hypotheses for the main Bohr theorem.  They are contrasting specializations
of the single asymptotic formula \eqref{eq:bohr-general-asymptotic}: the first
describes interaction orders above the logarithmic scale, where the factor
$N^{1/(2d)}$ disappears asymptotically, while the second describes the
logarithmic scale itself, where that factor leaves the explicit correction
$e^{1/(2\tau)}$.  The condition $d=o(N)$ is common to both regimes.  The sublogarithmic range
$d=o(\log N)$ is also covered by \eqref{eq:bohr-general-asymptotic}; in that
range the factor $N^{1/(2d)}$ must be retained.
\end{remark}

\appendix
\section{Potts hypercontractivity: normalization and proof of Proposition~\ref{prop:Potts-body}}\label{app:potts-details}

{\color{black}
The purpose of this appendix is only to justify, with the normalization used in
this paper, Proposition~\ref{prop:Potts-body}.  The general ingredients are
classical.  We use the sharp log--Sobolev constant for the complete graph,
computed by Diaconis and Saloff-Coste \cite{DSC}, its standard tensorization to
product spaces, and Gross's equivalence between log--Sobolev inequalities and
hypercontractivity \cite{Gross}.  The same complete-graph/Potts normalization
also appears in Gu and Polyanskiy \cite{GuPolyanskiy}.  Rather than reproducing
these general results, we record the short chain of identities needed to
translate them into the parameter $\rho$ used in Section~\ref{sec:potts}.

\subsection{The one-coordinate operator and its continuous-time normalization}

For $0\leq\rho\leq1$, recall from \eqref{eq:Potts-operator} that
\[
 (T_\rho g)(a)
 =\rho g(a)+(1-\rho)\int_{C_q}g\,d\mu_q.
\]
It is useful to make the associated transition kernel explicit.  Define
\[
 K_\rho:C_q\times C_q\longrightarrow[0,1]
\]
by
\[
 K_\rho(a,b)
 :=\rho\mathbf 1_{\{a\}}(b)+\frac{1-\rho}{q}.
\]
For each fixed $a\in C_q$,
\[
 \sum_{b\in C_q}K_\rho(a,b)=1,
\]
so $K_\rho(a,\cdot)$ is a probability measure on $C_q$, and
\[
 (T_\rho g)(a)=\sum_{b\in C_q}K_\rho(a,b)g(b).
\]
Equivalently,
\begin{equation}\label{eq:appendix-kernel-short}
 K_\rho(a,b)
 =
 \begin{cases}
 \displaystyle \frac1q+\frac{q-1}{q}\rho,&a=b,\\[2mm]
 \displaystyle \frac{1-\rho}{q},&a\neq b.
 \end{cases}
\end{equation}

Let $\Pi$ denote the averaging projection
\[
 \Pi g:=\left(\int_{C_q}g\,d\mu_q\right)\mathbf 1.
\]
Then
\begin{equation}\label{eq:appendix-Trho-Pi-short}
 T_\rho=\Pi+\rho(I-\Pi).
\end{equation}
Now consider the continuous-time complete-graph generator
\[
 \mathcal L:=-\frac q{q-1}(I-\Pi),
 \qquad P_t:=e^{t\mathcal L}.
\]
Since $\Pi$ and $I-\Pi$ are complementary projections,
\begin{equation}\label{eq:appendix-Pt-Trho-short}
 P_t
 =\Pi+e^{-qt/(q-1)}(I-\Pi)
 =T_{\rho(t)},
 \qquad
 \rho(t):=e^{-qt/(q-1)}.
\end{equation}
For $N$ coordinates we use the product semigroup
$P_t^{\otimes N}$, whose generator is the sum of the one-coordinate
generators.  Formula \eqref{eq:appendix-Pt-Trho-short} gives
\[
 P_t^{\otimes N}=T_{\rho(t)}^{\otimes N}.
\]

\subsection{The log--Sobolev input and Gross's theorem}

Write $\mathcal E$ for the Dirichlet form associated with $\mathcal L$.  The
sharp complete-graph log--Sobolev inequality of Diaconis--Saloff-Coste
\cite{DSC}, translated to the normalization above, is
\begin{equation}\label{eq:appendix-LSI-short}
 \operatorname{Ent}_{\mu_q}(h^2)
 \leq \frac1{\lambda_q}\,\mathcal E(h,h),
\end{equation}
where
\begin{equation}\label{eq:appendix-lambda-short}
 \lambda_2=1,
 \qquad
 \lambda_q=\frac{q-2}{(q-1)\log(q-1)}\quad(q\geq3).
\end{equation}
The constant is sharp.  The standard tensorization property of entropy and
log--Sobolev inequalities gives, with the same constant, the product estimate
\[
 \operatorname{Ent}_{\mu_q^N}(h^2)
 \leq \frac1{\lambda_q}\,\mathcal E^{(N)}(h,h),
\]
for every $N\geq1$; see, for instance, \cite{DSC}.  Gross's theorem
\cite{Gross} then yields, for real-valued functions,
\begin{equation}\label{eq:appendix-Gross-short}
 \|P_t^{\otimes N}f\|_2\leq\|f\|_p
 \qquad\text{whenever}\qquad
 e^{4\lambda_qt}\geq\frac1{p-1},
 \quad 1<p\leq2.
\end{equation}
{\color{black}The same estimate holds for complex-valued $f$.  Indeed,
$P_t^{\otimes N}$ is a positive Markov operator, hence
$|P_t^{\otimes N}f|\leq P_t^{\otimes N}|f|$ pointwise; applying the real
inequality to the nonnegative function $|f|$ gives the complex estimate with
the same constant.}
We record the conversion to the noise parameter because this is the only
normalization-sensitive point used later.  Put
\[
 t_p:=\frac1{4\lambda_q}\log\frac1{p-1}.
\]
Then \eqref{eq:appendix-Pt-Trho-short} gives
\begin{align*}
 \rho(t_p)
 &=\exp\left(-\frac q{q-1}t_p\right)\\
 &=\exp\left(
 -\frac{q}{4(q-1)\lambda_q}\log\frac1{p-1}
 \right)\\
 &=(p-1)^{\gamma_q},
\end{align*}
where
\[
 \gamma_q:=\frac{q}{4(q-1)\lambda_q}
 =
 \begin{cases}
 \frac12,&q=2,\\[1mm]
 \displaystyle\frac{q\log(q-1)}{4(q-2)},&q\geq3,
 \end{cases}
\]
which is exactly the value fixed in \eqref{eq:gamma-cq-intro}.  Hence
\eqref{eq:appendix-Gross-short} gives
\begin{equation}\label{eq:appendix-rho0-short}
 \|T_{(p-1)^{\gamma_q}}^{\otimes N}f\|_2\leq\|f\|_p.
\end{equation}
If $0\leq\rho\leq(p-1)^{\gamma_q}$, write
$\rho=\eta(p-1)^{\gamma_q}$ with $0\leq\eta\leq1$.  On the Fourier basis,
$T_\eta^{\otimes N}$ multiplies $x^\alpha$ by $\eta^{s(\alpha)}$, and
therefore Parseval gives
\[
 \|T_\eta^{\otimes N}g\|_2\leq\|g\|_2.
\]
Combining this with \eqref{eq:appendix-rho0-short} proves
\[
 \|T_\rho^{\otimes N}f\|_2\leq\|f\|_p
 \qquad
 \bigl(0\leq\rho\leq(p-1)^{\gamma_q}\bigr),
\]
which is the first assertion of Proposition~\ref{prop:Potts-body}.

\subsection{The low-interaction consequence}

We finally derive the form used in the recurrence.  Suppose
$\widehat f(\alpha)=0$ whenever $s(\alpha)>k$, and set
$\rho_0=(p-1)^{\gamma_q}$.  By Parseval and the Fourier multiplier formula,
\begin{align*}
 \|T_{\rho_0}^{\otimes N}f\|_2^2
 &=\sum_{s(\alpha)\leq k}
   \rho_0^{2s(\alpha)}|\widehat f(\alpha)|^2\\
 &\geq \rho_0^{2k}
   \sum_{s(\alpha)\leq k}|\widehat f(\alpha)|^2
 =\rho_0^{2k}\|f\|_2^2.
\end{align*}
Together with \eqref{eq:appendix-rho0-short}, this yields
\[
 \rho_0^k\|f\|_2\leq\|f\|_p,
\]
and therefore
\[
 \|f\|_2
 \leq(p-1)^{-\gamma_qk}\|f\|_p.
\]
This is the second assertion of Proposition~\ref{prop:Potts-body}, and
Corollary~\ref{cor:Potts-factor-body} follows by substituting
$p=p_m=2m/(m+1)$ and $k=d-m$.

The appendix has deliberately retained only the normalization-dependent
steps.  The complete-graph log--Sobolev theorem, tensorization, and the
Gross implication are invoked from the standard references above; the
calculations displayed here are precisely those needed to identify our
$T_\rho$ with the continuous-time semigroup and to recover the exponent
$\gamma_q$ appearing in the recurrence.
}

\section*{Author contributions}
D.M.P. proposed the problem and initiated the project.  A.R.J. developed the
main arguments and worked on the proofs.  Both authors discussed the results,
checked the mathematical details, and contributed to the writing and revision
of the manuscript.

\section*{Funding}
D.M.P. was partially supported by CNPq through Grants
406457/2023-9, 305807/2025-0, and 403964/2024-5.  The funding agency had no
role in the development of the results, the preparation of the manuscript,
or the decision to submit the article.

\section*{Data availability}
No data were used for the research described in this article.

\section*{Declaration of competing interest}
The authors declare that they have no known competing financial interests or
personal relationships that could have appeared to influence the work
reported in this article.

\section*{Declaration of generative AI and AI-assisted technologies in the manuscript preparation process}
During the preparation of this work, the authors used ChatGPT (OpenAI) to
assist with exploratory proof development, the organization and exposition
of arguments, literature searches, and consistency checks.  The authors
reviewed and verified the mathematical content, edited the manuscript, and
take full responsibility for its content.

\end{document}